\documentclass[reqno, 10pt]{amsart}
\title[An explicit isomorphism between quantum and classical $\sl_n$]
{An explicit isomorphism between\\ quantum and classical $\sl_{n}$}
\author[A. Appel]{Andrea Appel}
\address{School of Mathematics,
University of Edinburgh,
James Clerk Maxwell Building, 
Peter Guthrie Tait Road,
Edinburgh, EH9 3FD, UK}
\email{andrea.appel@ed.ac.uk}

\author[S. Gautam]{Sachin Gautam}
\address{Department of Mathematics, The Ohio State University,
Columbus, OH 43210 (USA)
}
\email{gautam.42@osu.edu}

\usepackage{amsmath,amssymb}
\usepackage[usenames,dvipsnames]{color}
\usepackage{hyperref} 
\hypersetup{
    colorlinks=true,                            
    linkcolor=blue, 
    citecolor=blue,                             
    filecolor=blue,                                     
    urlcolor=blue                               
}

\usepackage[all]{xy}
\xyoption{arc}
\newtheorem*{thm}{Theorem}
\newtheorem*{prop}{Proposition}
\newtheorem*{lem}{Lemma}
\newtheorem*{cor}{Corollary}
\newtheorem*{conj}{Conjecture}

\newenvironment{pf}{\paragraph{{\sc Proof}}}{\qed\par\medskip}
\theoremstyle{definition}

\numberwithin{equation}{section}

\numberwithin{figure}{section}

\newcommand{\phiCP}{\varphi_{\scriptstyle{\text{CP}}}}

\newcommand{\Tsl}{\operatorname{T}} 

\newcommand{\Psl}{\operatorname{P}} 

\newcommand{\roots}[2]{\zeta^{(#1)}_{#2}} 

\newcommand{\commelt}[2]{\zed_{#2}^{(#1)}} 

\newcommand{\fkS}{\mathfrak{S}} 

\newcommand{\hhalf}{\frac{\hbar}{2}}

\newcommand{\qmin}[3]{\Delta^{#1}_{#2}\lp\Tsl\rp\lp #3\rp}

\newcommand{\psiop}[2]{\psi^{(#1)}\left[#2\right]}
\newcommand{\psiopw}[1]{\psi^{(#1)}}
\newcommand{\psiopww}{\psi}

\newcommand{\qminpsi}[4]{\Delta^{#2}_{#3}\lp\psiopw{#1} \rp\lp #4\rp}
\newcommand{\qminpsiw}[3]{\Delta^{#1}_{#2}\lp\psiopww \rp\lp #3\rp}
\newcommand{\tqmin}[3]{\Delta^{#1}_{#2}(\phi)\lp #3\rp}

\newcommand{\ten}{\otimes}

\newcommand{\semi}[1]{\mathsf{SC}(#1)}

\newcommand{\bs}[1]{\left| #1\right\rangle}
\newcommand{\elm}[2]{\left| #1\right\rangle\left\langle #2\right|}

\newcommand{\aux}[1]{\operatorname{End}(\C^{#1})}
\newcommand{\auxaux}[1]{\operatorname{End}(\C^{#1}\otimes\C^{#1})}
\newcommand{\haux}[2]{\operatorname{End}\lp \lp \C^{#1}\rp^{\otimes #2}\rp}

\newcommand{\RTT}{\operatorname{RTT}}

\newcommand{\cowt}{\varpi^{\vee}}
\newcommand{\Tr}[1]{\operatorname{Tr}_{#1}}

\newcommand{\id}{\mathbf{1}}

\newcommand{\veps}{\varepsilon}
\newcommand{\lp}{\left(}
\newcommand{\rp}{\right)}

\newcommand{\lb}{\left[}
\newcommand{\rb}{\right]}

\newcommand{\g}{\mathfrak{g}}
\newcommand{\h}{\mathfrak{h}}
\newcommand{\gl}{\mathfrak{gl}}

\newcommand{\Y}{$\mathcal{Y}$}

\newcommand{\bfA}{\mathbf{A}}

\newcommand{\C}{\mathbb{C}}

\newcommand{\N}{\mathbb{Z}_{\geqslant 0}}
\newcommand{\Q}{\mathbb{Q}}

\newcommand{\Z}{\mathbb{Z}}

\newcommand {\wh}[1]{\widehat{#1}}
\newcommand {\ol}[1]{\overline{#1}}
\newcommand {\ul}[1]{\underline{#1}}

\newcommand{\End}{\operatorname{End}}

\newcommand{\Ad}{\operatorname{Ad}}
\newcommand{\ev}{\mathsf{ev}}
\newcommand{\evcp}{\mathsf{ev}_{\scriptstyle{\operatorname{CP}}}}

\newcommand {\aand}{\qquad\text{and}\qquad}

\newcommand {\lhs}{left--hand side }
\newcommand {\rhs}{right--hand side }
\newcommand {\wrt}{with respect to }

\newcommand{\ds}{\displaystyle}
\newcommand{\wt}[1]{\widetilde{#1}}

\newcommand{\zed}{\mathfrak{z}}

\newcommand {\Omit}[1]{}

\newcommand{\qdet}{\operatorname{qdet}}

\newcommand {\Uhg}{U_\hbar\g}

\newcommand{\Uhsl}[1]{U_{\hbar}(\Lsl_{#1})}
\newcommand {\Yhg}{Y_\hbar(\g)}

\newcommand {\Yhgl}[1]{Y_\hbar(\gl_{#1})}
\newcommand{\Yhsl}[1]{Y_{\hbar}(\sl_{#1})}

\newcommand {\Usl}[1]{U(\Lsl_{#1})}
\newcommand {\isom}{\stackrel{\sim}{\rightarrow}}

\renewcommand {\gl}{\mathfrak{gl}}
\renewcommand {\sl}{\mathfrak{sl}}

\newcommand {\Id}{\operatorname{Id}}
\newcommand {\ie}{{\it i.e.}, }

\newcommand{\bfI}{{\mathbf I}}

\newcommand{\hext}[1]{{#1}[\negthinspace[\hbar]\negthinspace]}

\renewcommand{\Uhsl}[1]{U_{\hbar}\sl_{#1}}
\renewcommand{\Usl}[1]{U\sl_{#1}}
\renewcommand{\Yhsl}[1]{Y_{\hbar}\sl_{#1}}
\renewcommand{\Yhgl}[1]{Y_{\hbar}\gl_{#1}}
\renewcommand{\id}{\mathsf{id}}
\newcommand{\fmls}{\hext{\C}}
\newcommand{\Ugh}{\hext{U\g}}


\begin{document}

\begin{abstract}
Let $\g$ be a complex semisimple Lie algebra. Although the quantum group $\Uhg$ is known
to be isomorphic, as an algebra, to the undeformed enveloping algebra $\Ugh$, no such
isomorphism is known when $\g\neq \sl_2$. In this paper, we construct an explicit isomorphism
for $\g = \sl_n$, for every $n\geqslant 2$, which preserves the standard flag of type $\mathsf{A}$.
We conjecture that this isomorphism quantizes the Poisson diffeomorphism of Alekseev and Meinrenken
\cite{am-gz}.
\end{abstract}

\maketitle
\setcounter{tocdepth}{1}
\tableofcontents

\section{Introduction}\label{sec:intro}

\subsection{}
Let $\g$ be a complex semisimple Lie algebra.
The Drinfeld--Jimbo quantum group $\Uhg$ is a topological Hopf algebra over 
the ring of formal power series $\fmls$, which deforms
the universal enveloping algebra $U\g$ \cite{drinfeld-qybe,jimbo-qg, jimbo-R-matrix}. 
In \cite{drinfeld-almost}, Drinfeld pointed out that the 
algebra structure of $\Ugh$ remains unchanged under quantization, 
\ie there exists an isomorphism of 
$\fmls$--algebras $\psi:\Uhg\to\Ugh$, which is congruent to the identity modulo $\hbar$.
Due to its cohomological origin, such an isomorphism is highly non--canonical and 
indeed unknown with the sole exception of $\sl_2$ (e.g. \cite[\S 5]{drinfeld-almost} and 
\cite[Prop.~6.4.6]{chari-pressley}).

In this paper, we contruct an explicit algebra isomorphism 
$\varphi : \Uhsl{n}\to \hext{\Usl{n}}$ for any $n\geqslant2$.  
We refer the reader to Section \ref{sec:main}, Theorem \ref{thm:main1}, 
for the formulae defining $\varphi$.
Here, we will explain how the isomorphism is obtained
and state some of its properties.

\subsection{}

Our construction relies on the homomorphism between the quantum loop algebra
and the Yangian of $\sl_n$ from \cite{sachin-valerio-1}. Namely, $\varphi$ is defined
to be the following composition
\[
\xymatrix@C=2cm{
U_\hbar L\sl_{n} \ar[r]^{\Phi} & \wh{\Yhsl{n}} \ar[d]^{\ev}\\
\Uhsl{n} \ar@{^(->}[u] \ar[r]_{\varphi}& \hext{\Usl{n}}
}
\]
where 
\begin{itemize}
\item $\Phi$ is the algebra homomorphism between the quantum loop algebra
and (an appropriate completion of) the Yangian of $\sl_n$, 
defined and studied by the second author and V. Toledano Laredo in \cite{sachin-valerio-1}.
While the results of \cite{sachin-valerio-1} hold for any semisimple Lie algebra $\g$,
in this paper we only need them for $\sl_n$, and even further, only the restriction
of $\Phi$ to $\Uhsl{n}$. We refer the reader to Section \ref{sec:ysln}
for a brief review of \cite{sachin-valerio-1}.

\item $\ev$ is the evaluation homomorphism at $0$. 
It is well known that the evaluation homomorphism exists only in type $\mathsf{A}$ (see
\cite[Prop.~12.1.15]{chari-pressley} for both these statements).
\end{itemize}

The homomorphism $\Phi$ in the diagram above is given explicitly in
the loop presentation of Yangian (also known as \emph{Drinfeld's new presentation}
\cite{drinfeld-yangian-qaffine}), while the evaluation 
homomorphism $\ev$ is known only in either the $J$--presentation,
or the $RTT$ presentation of Yangian (see, for example, \cite[Chapter 12]{chari-pressley}
and \cite[Chapter 1]{molev-yangian}). 
Thus, in order to work out an explicit formula for $\varphi$, defined via the diagram
above, we need to rewrite $\ev$ in terms of the Drinfeld currents
of $\Yhsl{n}$.

\subsection{}
One of the main results of this
paper, Theorem \ref{thm:evishom}, achieves exactly this. 
A large part of this paper is devoted towards proving Theorem \ref{thm:evishom},
but we can briefly explain the idea behind it, which could be of independent interest.\\

We construct a matrix $\Tsl \in \End(\C^n)\otimes \Usl{n}$ such that $\Tsl(u):= u\Id - \hbar \Tsl$
satisfies the $\RTT$ relations with respect to the Yang's $R$--matrix $R(u) := u\Id + \hbar P \in 
\End(\C^n\otimes\C^n)[u,\hbar]$ (here $P$ is the flip operator). See Section \ref{ssec:transfer} 
and Proposition \ref{pr:rtt} for these assertions. This observation allows us to work out the 
commutation relations between the quantum minors of $\Tsl(u)$ using the standard techniques 
of $\RTT$ algebras (Proposition \ref{pr:minor-comm}). Morally speaking, $\RTT$ relations
make sure that linear algebra carries over to the matrix $\Tsl(u)$, even though its
entries are from a non--commutative ring. As a side consequence, one obtains another
proof of a few well--known results from the classical invariant theory of $\sl_n$, 
see Proposition \ref{pr:bethe}.
We would like to refer to \cite[Chapter 1]{molev-yangian} for a more thorough
treatment, which also inspired most of our proofs in Section \ref{sec:qminsl}. The
sole exception being Proposition \ref{pr:psi-operator}, which carries
out an important reduction (to rank 1 and 2) step. Our proof of this result is 
new and seems to be applicable to more general situations.

\subsection{}

We now highlight some of the properties of $\varphi: \Uhsl{n}\to \hext{\Usl{n}}$,
which are clear from the formulae given in Section \ref{ssec:iso}.

\begin{enumerate}
\item $\varphi$ is defined over $\Q$ (see Remark \ref{ssec:main1-rems} (ii)).
\item $\varphi$ preserves the standard flag of Levi subalgebras in $\sl_{n}$. 
More precisely, let $k\in \{1,\ldots,n-1\}$ and identify $\Uhsl{k+1}$ with the 
subalgebra of $\Uhsl{n}$ generated by $\{K_i,E_i,F_i\}_{1\leq i\leq k}$ 
(see Section \ref{ssec:qsl}). Similarly, let $\sl_{k+1}$ be viewed as a Lie subalgebra of 
$\sl_n$, corresponding to matrices supported on the top--left corner of order $k+1$. 
Then $\displaystyle \varphi\lp \Uhsl{k+1}\rp \subset \hext{U(\h+\sl_{k+1})}$, where
$\h\subset\sl_n$ is the Cartan subalgebra of diagonal matrices. The appearance of
the Levi subalgebra $\h+\sl_{k+1}$ is necessary since the diagonal entries of
$\Tsl$ are supported on $\h$ and do not respect the standard embedding of $\sl_{k+1}$
in $\sl_n$ (cf. Section \ref{ssec:transfer}).
\item $\varphi$ has a {\em semi--classical limit}, as we explain in the next paragraph.
\end{enumerate}

\subsection{}
The identification of $\Uhg$ and $\Ugh$ as algebras has a more geometric 
interpretation as a quantization of the (formal) Ginzburg--Weinstein linearization 
theorem. 
In \cite{ginzburg-weinstein}, Ginzburg and Weinstein proved 
that, for any compact Lie group $K$ with its standard Poisson structure, the dual 
Poisson Lie group $K^*$ is Poisson isomorphic to the dual Lie algebra 
$\mathfrak{k}^*$, with its canonical linear (Kostant-Kirillov-Souriau) Poisson 
structure. Subsequently, this result has been generalized in the works of Alekseev and
Meinrenken \cite{alekseev,am-gz,alekseev-meinrenken} and Boalch \cite{boalch-stokes}.
Finally, in \cite{enriquez-etingof-marshall}, Enriquez, 
Etingof and Marshall provided a construction, for any
finite--dimensional quasitriangular Lie bialgebra $\g$, of 
a family of formal Poisson isomorphisms
between the Poisson manifolds $\g^*$ and $G^*$. 
We refer to such maps as \emph{formal dual 
exponential maps}.\\

Formal dual exponential  maps can be \emph{quantized} by certain algebra isomorphisms 
$\psi:\Uhg\to\Ugh$ as we explain now. Let $B$ be a quantized universal enveloping algebra (QUE), \ie
an $\hbar$--adically complete Hopf algebra over $\hext{\C}$,
together with a fixed isomorphism of Hopf algberas
$B/\hbar B\isom U\mathfrak{a}$, where $\mathfrak{a}$
is a Lie algebra over $\C$.
According to the \emph{quantum duality principle} of Drinfeld and 
Gavarini \cite{drinfeld-quantum-groups, gavarini}, any QUE $B$ contains a
{\em quantized formal series Hopf subalgebra} (QFSH)
\[
B':= \{b\in B\ |\ \forall\ n\geq 1,\ p_n(b)\in\hbar^n B^{\otimes n}\}\subset B,
\]
where, $p_n := (\id-\iota\varepsilon)^{\otimes n} \circ \Delta^{(n)}
: B\to B^{\otimes n}$. Here, $\Delta^{(n)}$ is the $n^{\text{th}}$
iterated coproduct, $\varepsilon:B\to\hext{\C}$
is the counit, and $\iota:\hext{\C}\to B$ is the unit.\\

The semi--classical limit of $B$ is then defined as $\semi{B}
:= B'/\hbar B'$. We say that an algebra homomorphism 
$\psi:A\to B$ between two QUEs admits a semiclassical limit if 
it respects the QFSH $A'$ and $B'$ and therefore it descends to a homomorphism
$\semi{\psi}:\semi{A}\to\semi{B}$. Alternatively, we say that $\psi$
is a quantization of $\semi{\psi}$.\\

It is easy to see that, in the case of the QUEs $\Uhg$ and $\Ugh$, 
one has (cf. \cite{gavarini})
\[
\semi{\Uhg} \cong \mathcal{O}_{G^*} \aand
\semi{\Ugh} \cong \C[\negthinspace[\g^*]\negthinspace]=:\wh{\mathcal{O}}_{\g^*}.
\]
Therefore, any isomorphism of algebras $\psi:\Uhg\to\Ugh$ preserving the QFSHs
$(\Uhg)'$ and $(\Ugh)'$ is a \emph{quantized dual exponential map}, \ie it has a 
semi--classical limit and it gives rise to a formal dual exponential map $\g^*\to G^*$. 

\subsection{}
In the case of $\g=\sl_n$, the isomorphism $\varphi$, which we construct in Theorem \ref{thm:main1}, 
is a quantized dual exponential map and $\semi{\varphi}$ gives rise to a dual exponential map 
$\sl_n^*\cong\sl_n\to SL_n^*$. In \cite[Thm. 1.2]{am-gz}, Alekseev and Meinrenken show that such 
Poisson morphism exists and it is uniquely determined by certain properties, which are the analogues 
of (1) and (2) listed above. One verifies directly that for $n=2$ $\semi{\varphi}$ coincides with the
Alekseev--Meinrenken map. This suggest the following
\begin{conj}\label{iconj: sc=am}
The isomorphism $\varphi$ is a canonical quantization of the Alekseev--Mein-renken dual
exponential map $\sl_n^*\to SL_n^*$ \cite[Thm.1.2]{am-gz}.
\end{conj}
We will return to this conjecture in \cite{andrea-sachin-2}.

\subsection{}
Finally, it is worth mentioning that in the case $n=2$ the isomorphism $\varphi$
is easily seen to identify (up to a sign) the action of the quantum Weyl group operator of $\Uhsl{2}$
with the exponential of the Casimir element of $U\sl_2$. This simple observation provides a direct
and straightforward proof that the monodromy of the Casimir connection is computed by the quantum 
Weyl group operators of the quantum group (cf. \cite{atl0, valerio0}). A further study of the isomorphism
$\varphi$ could lead to a similar result for $\sl_n$, $n\geqslant 2$, providing a direct proof
which does not rely on any cohomological results, in contrast with the methods used in \cite{atl1, valerio1}.

\subsection{Outline of the paper}
In Section \ref{sec:main}, we recall the basic definitions of the enveloping algebra
$U\sl_n$, the quantum group  $U_{\hbar}\sl_n$ and 
describe the isomorphism between 
$U_{\hbar}\sl_n$ and $\hext{U\sl_n}$ in Theorem \ref{thm:main1}.
We then 
discuss the case of $\sl_2$ and we give a direct proof 
that the proposed map is an algebra homomorphism.
In Section \ref{sec:ysln}, we review the definition of the Yangian $\Yhsl{n}$
and the main construction of \cite{sachin-valerio-1}, yielding an algebra 
homomorphism from $\Uhsl{n}$ to the completion of $\Yhsl{n}$ \wrt its $\N$--grading. 
In Section \ref{sec:qminsl}, we study the matrix $\Tsl(u) \in \Usl{n}[\hbar,u]$
introduced in Section \ref{sec:main}
and the relations satisfied by its quantum minors. In particular, we show that $\Tsl(u)$ 
satisfies the 
$RTT=TTR$ relation (Proposition \ref{pr:rtt}) which leads to an analogue of 
the Capelli identity for $\sl_n$. That is,
the coefficients of the quantum--determinant of $\Tsl(u)$ are algebraically independent and 
generate the center of $U\sl_n$ (see Proposition \ref{pr:bethe}).
In Section \ref{sec:eval}, the determinant identities obtained in the previous 
section are used to construct 
the evaluation homomorphism from $\Yhsl{n}$ to ${U\sl_n}[\hbar]$ in the loop
presentation of the Yangian (Theorem \ref{thm:evishom}). This is the
last step in the proof of Theorem \ref{thm:main1}.


\subsection{Acknowledgements} 
We are very grateful to Pavel Etingof for suggesting this problem to us, 
and to Valerio Toledano Laredo for patiently explaining to us the semi--classical
limits and pointing us to the literature on Poisson geometry.
We are also thankful to Maxim Nazarov, Alexander Tsymbaliuk, and Alex Weekes  
for their useful comments and suggestions. 

This research was supported in part by Perimeter Institute for Theoretical Physics.
Research at Perimeter Institute is supported by the Government of Canada through 
the Department of Innovation, Science and Economic Development and by the 
Province of Ontario through the Ministry of Research and Innovation.
A.~A. was partially supported by the ERC grant
STG--637618 and the NSF grant DMS--1255334. S.~G. acknowledges the generous support
of the Simons Foundation, in the form of a collaboration grant for mathematicians, number
526947.


\section{The isomorphism between $\Uhsl{n}$ and $\hext{U\sl_n}$}\label{sec:main}

In this section, we recall the definition of the enveloping algebra
$U\sl_n$ and the quantum group $U_{\hbar}\sl_n$. We state the
main theorem, describing the isomorphism between them. As an 
example, we prove the case of $\sl_2$ by direct computation.

\subsection{Notations}\label{ssec:An-notations}

Let $n\in\Z_{\geqslant 2}$ and let $\bfI = \{1,\ldots,n-1\}$. Let $\bfA = (a_{ij})_{i,j\in\bfI}$
be the Cartan matrix of type $\mathsf{A}_{n-1}$. Namely,
\[
a_{ij} = \left\{\begin{array}{ccl} 2 & \text{if} & i=j \\
-1 & \text{if} & |i-j|=1 \\ 0 & \text{if} & |i-j|>1 \end{array}\right.
\]
Throughout this paper, we consider $\hbar$ to be a formal variable and 
set $q = e^{\hhalf} \in \fmls$.

\subsection{Quantum group}\label{ssec:qsl}

$\Uhsl{n}$ is a unital associative algebra over $\fmls$ (topologically) generated by
$\{H_i, E_i, F_i\}_{i\in\bfI}$ subject to the following list of relations:
\begin{itemize}
\item[(QG1)] For each $i,j\in\bfI$
\[
[H_i,H_j]=0;
\]
\item[(QG2)] For each $i,j\in\bfI$, we have
\[
[H_i,E_j]=a_{ij}E_j \qquad\mbox{and}\qquad [H_i,F_j]=-a_{ij}F_j;
\]
\item[(QG3)] For each $i,j\in\bfI$, we have
\[
[E_i,F_j] = \delta_{ij} \frac{K_i-K_i^{-1}}{q-q^{-1}},
\]
where we set $K_i = e^{\hbar H_i/2}$.
\item[(QG4)] For $i\neq j$, we have $[E_i,E_j]=0=[F_i,F_j]$ if $a_{ij}=0$. If $a_{ij}=-1$:
\[
E_i^2E_j - (q+q^{-1})E_iE_jE_i + E_jE_i^2=0, 
\]
\[
F_i^2F_j - (q+q^{-1})F_iF_jF_i + F_jF_i^2=0.
\]
\end{itemize}

$\Uhsl{n}$ has a structure of Hopf algebra, with coproduct and counit given,
respectively, by
\begin{align*}
\Delta(H_i) &= H_i\otimes 1 + 1\otimes H_i ,\\
\Delta(E_i) &= E_i\otimes K_i + 1\otimes E_i ,\\
\Delta(F_i) &= F_i\otimes 1  + K_i^{-1}\otimes F_i,
\end{align*}
and $\veps(H_i) = \veps{E_i} = \veps(F_i) = 0$.

\subsection{Universal enveloping algebra of $\sl_n$}\label{ssec:usl}

Recall that $\Usl{n}$ is a unital associative algebra over $\C$
generated by $\{h_i, e_{kl}\}_{1\leqslant i\leqslant n-1, 1\leqslant k\neq l\leqslant n}$ subject
to the following relations: $[h_i,h_j]=0$, 
$[h_i,e_{kl}] = (\delta_{ik}-\delta_{il}-\delta_{i+1,k}+\delta_{i+1,l})e_{kl}$,
and $[e_{k,l},e_{k',l'}] = \delta_{l,k'}e_{k,l'} - \delta_{k,l'}e_{k',l}$,
where we understand that $h_i = e_{i,i}-e_{i+1,i+1}$. Thus we have
$e_{kk}-e_{ll} = h_k + \cdots + h_{l-1}$, for $1\leqslant k<l\leqslant n$.

Let $\h$ be the span of $\{h_i\}_{1\leqslant i\leqslant n-1}$. The standard bilinear form 
defined by the Cartan matrix of $\sl_n$ on $\h$ is given by $(h_i,h_j) = a_{ij}$. 
With respect to $(\cdot,\cdot)$, we consider the fundamental coweights $\cowt_i \in\h$ 
defined by $(\cowt_i,h_j) = \delta_{ij}$, so that
\[
\cowt_i = \frac{1}{n}\lp (n-i)\sum_{j=1}^{i-1} jh_j + i\sum_{j=i}^{n-1} (n-j)h_j\rp.
\]

\subsection{$\operatorname{T}$--matrix}\label{ssec:transfer}
Let $\Tsl = (\Tsl_{ij})_{1\leqslant i,j\leqslant n}$ be $n\times n$ matrix with entries from $\Usl{n}$
defined as
\footnote{In \cite{dskv17}, De Sole--Kac--Valeri introduced a more general version of the 
matrix $T$, which includes the classical Lie algebras $\mathfrak{so}_n$ and $\mathfrak{sp}_n$.}
:
\[
\Tsl_{ij}=\left\{
\begin{array}{ccc}
\cowt_i-\cowt_{i-1} & \mbox{if} & i=j\\
e_{ij} & \mbox{if} & i\neq j\\
\end{array}
\right.
\]

Here we assume that $\cowt_{0} = \cowt_n = 0$. Note that the diagonal entries
are uniquely determined by the requirement that $\sum_{i=1}^n \Tsl_{ii} = 0$ and
that for every $1\leqslant k<l\leqslant n$:
\begin{equation}\label{eq:X-diagonal}
\Tsl_{kk} - \Tsl_{ll} = e_{kk}-e_{ll} = h_k+\cdots +h_{l-1}.
\end{equation}

Define $\Tsl(u) := u\Id - \hbar\Tsl$. Given $1\leqslant m\leqslant n$ and $\ul{a} = (a_1,\ldots,a_m)$
and $\ul{b} = (b_1,\ldots, b_m)$ elements of $\{1,\ldots,n\}^m$, we consider the quantum--minor
of $\Tsl(u)$, $\qmin{\ul{a}}{\ul{b}}{u}\in\Usl{n}[u,\hbar]$, defined as:
\begin{equation}\label{eq:qminsl}
\qmin{\ul{a}}{\ul{b}}{u} := \sum_{\sigma\in\fkS_m} (-1)^{\sigma}
\Tsl_{a_{\sigma(1)},b_1}(u_1)\cdots \Tsl_{a_{\sigma(m)},b_m}(u_m),
\end{equation}
where $u_j = u + \hbar(j-1)$.
The quantum--minors of $\Tsl(u)$ are studied in Section \ref{sec:qminsl}.
For each $1\leqslant k\leqslant n$, let $\Psl_k(u)$ be the principal $k\times k$
quantum--minor $\ds\Psl_k(u)=\qmin{1,\ldots,k}{1,\ldots,k}{u - \hhalf(k-1)}$.
Thus,
\begin{multline}\label{eq:polyP}
\Psl_k(u) := \sum_{\sigma\in\fkS_k} (-1)^{\sigma}
\Tsl_{\sigma(1),1}\lp u - \hhalf(k-1)\rp
\Tsl_{\sigma(2),2}\lp u - \hhalf(k-3)\rp \\ \cdots
\Tsl_{\sigma(k),k}\lp u + \hhalf(k-1)\rp.
\end{multline}

We prove in \S\ref{ssec:bethe} that the subalgebra generated in $\Usl{n}$
by the coefficients appearing in $\Psl_k(u)$, $1\leqslant k\leqslant n$, 
is maximal commutative and we denote by
$\roots{k}{1},\ldots,\roots{k}{k}$ the roots of $\Psl_k(u)$ defined 
in an appropriate splitting extension.

\subsection{Isomorphism $\varphi$}\label{ssec:iso}

Choose two formal series $G^{\pm}(x)\in 1+x\C[\negthinspace[x]\negthinspace]$ satisfying the following
two conditions:
\begin{equation}\label{eq:factorizationproblem}
\begin{split}
G^-(x) &= G^+(-x), \\
G^+(x)G^-(x) &= \frac{e^{x/2}-e^{-x/2}}{x}.
\end{split}
\end{equation}

Consider the following assignment $\varphi : \Uhsl{n}\to \hext{\Usl{n}}$.
\begin{itemize}
\item $\varphi(H_i) = h_i$ for each $i\in\bfI$.\\
\item For each $k\in\bfI$ we have
\begin{align}
\varphi(E_k) &= \frac{\hbar}{q-q^{-1}} \nonumber \\
& \sum_{i=1}^k
\lp
\frac{\prod_{a=1}^{k-1} G^+\lp\roots{k}{i}-\roots{k-1}{a}-\hhalf\rp
\prod_{b=1}^{k+1} G^+\lp\roots{k}{i} - \roots{k+1}{b}-\hhalf\rp}
{\prod_{c\neq i} G^+(\roots{k}{i} - \roots{k}{c})
G^+(\roots{k}{i} - \roots{k}{c}-\hbar)}
\rp \cdot \nonumber\\
& \ \ \ \lp
\sum_{j=1}^k (-1)^{k-j}
\frac{\qmin{1,\ldots,\wh{j},\ldots,k}{1,\ldots, k-1}
{\roots{k}{i}-\hhalf(k-1)}}
{\prod_{r\neq i} (\roots{k}{i} - \roots{k}{r})}
e_{j,k+1}
\rp \label{eq:phiEk},\\
\varphi(F_k) &= \frac{\hbar}{q-q^{-1}} \nonumber \\
& \sum_{i=1}^k
\lp
\frac{\prod_{a=1}^{k-1} G^-\lp\roots{k}{i}-\roots{k-1}{a}+\hhalf\rp
\prod_{b=1}^{k+1} G^-\lp\roots{k}{i} - \roots{k+1}{b}+\hhalf\rp}
{\prod_{c\neq i} G^-(\roots{k}{i} - \roots{k}{c})
G^-(\roots{k}{i} - \roots{k}{c}-\hbar)}
\rp \cdot \nonumber \\
& \ \ \ \lp
\sum_{j=1}^k (-1)^{k-j}
\frac{\qmin{1,\ldots, k-1}{1,\ldots,\wh{j},\ldots,k}
{\roots{k}{i}-\hhalf(k-3)}}
{\prod_{r\neq i} (\roots{k}{i} - \roots{k}{r})}
e_{k+1,j} \label{eq:phiFk}
\rp,
\end{align}
\end{itemize}
where $\wh{j}$ means that the index $j$ is omitted.\\

\begin{thm}\label{thm:main1}
The assignment $\varphi$ given above is an isomorphism of algebras
$\varphi:\Uhsl{n}\isom\hext{\Usl{n}}$, which satisfies $\varphi|_{\h}=\id_{\h}$ and
$\varphi=\id\mod\hbar$.
\end{thm}

\subsection{Remarks}\label{ssec:main1-rems}

\begin{enumerate}
\item 
The expressions given above of $\varphi(E_k)$
and $\varphi(F_k)$ belong to a splitting extension of $\{\Psl_j(u)\}_{1\leqslant j\leqslant n}$.
However, the proof of Theorem \ref{thm:main1} will highlight the fact that the 
right--hand sides of \eqref{eq:phiEk} and
\eqref{eq:phiFk} are symmetric in $\{\roots{j}{1},\ldots,\roots{j}{j}\}$ for each $j$,
and therefore live in $\hext{\Usl{n}}$.\\

\item For the formal series $G^{\pm}(x)$ satisfying \eqref{eq:factorizationproblem}, there
are two natural candidates. The first one,
used in \cite{sachin-valerio-1}, is $\ds G^{\pm}(x) = \lp \frac{e^{x/2}-e^{-x/2}}{x}\rp ^{\frac{1}{2}}$.
With this choice, the isomorphism $\varphi$ is defined over $\hext{\Q}$.
The second choice, implicitly used in \cite{sachin-valerio-2}, is 
$\ds G^{\pm}(x) = \frac{1}{\Gamma\lp 1\pm \frac{x}{2\pi\iota}\rp}$,
where $\Gamma$ is the Euler's gamma
function.
\end{enumerate}

\subsection{The case of $\sl_2$}\label{ssec:n=2}

For $n=2$, we have $\ds \Tsl(u) = \left[\begin{array}{cc} u-\hbar\cowt{} & -\hbar e_{12} \\
-\hbar e_{21} & u + \hbar\cowt{} \end{array}\right]$. Recall that here
$\cowt{} = h/2$. Thus we have 
$\Psl_1(u) = u-\hbar\cowt{}$ and 
\[
\Psl_2(u) = u^2 - \lp\hhalf\rp^2 (2C+1),
\]
where $C = e_{12}e_{21} + e_{21}e_{12} + h^2/2$ is the Casimir element of $\sl_2$.
As per our convention, we set $\Psl_0(u) = \Psl_3(u) = 1$. The roots of these
polynomials are: 
\[
\roots{1}{1} = \hbar\cowt{} \aand
\roots{2}{1}, \roots{2}{2} = \pm \hhalf \sqrt{2C+1}.
\]
Using Theorem \ref{thm:main1} we get the following
\begin{align*}
\varphi(E) &= \frac{\hbar}{q-q^{-1}} G^+\lp \hbar\cowt{} - \hhalf(1+\sqrt{2C+1})\rp
 G^+\lp \hbar\cowt{} - \hhalf(1-\sqrt{2C+1})\rp e_{12}, \\
\varphi(F) &= \frac{\hbar}{q-q^{-1}} G^-\lp \hbar\cowt{} + \hhalf(1+\sqrt{2C+1})\rp
 G^-\lp \hbar\cowt{} + \hhalf(1-\sqrt{2C+1})\rp e_{21}.
\end{align*}

Note that our isomorphism $\varphi$ differs from the one given in \cite[\S 6.4]{chari-pressley}.
To write their formulae, we have to make the following changes: the element $\ol{\Omega}$
from \cite[\S 6.4 {\bf B}]{chari-pressley} is $\ol{\Omega} = \frac{1}{4}(2C+1)$, and the deformation
parameter there denoted by $h$ is our $\hhalf$. With this in mind, the isomorphism
of \cite[Prop.~6.4.6]{chari-pressley}, denoted by $\phiCP$, given as follows:
$\phiCP(H) = h$, $\phiCP(F) = e_{21}$, and
\[
\phiCP(E) = 4\lp \frac{q^{\sqrt{2C+1}} + q^{-\sqrt{2C+1}} - q^{-1}K - qK^{-1}}
{(q-q^{-1})^2 (2C+2h-h^2)} \rp e_{12}.
\]

Though not essential, we give a direct proof that our $\varphi$ is an algebra homomorphism
for the $\sl_2$ case below, which is analogous to
the one in \cite[Prop.~6.4.6]{chari-pressley}.\\

The only non--trivial relation to verify is $\ds [\varphi(E),\varphi(F)] = \frac{K-K^{-1}}{q-q^{-1}}$,
where as usual we write $K = e^{\hhalf h}$. For this we use the fact that $C$ is central and
for any function $\mathcal{P}(\cowt{})$ we have $\mathcal{P}(\cowt{}-1) e_{12} = e_{12}\mathcal{P}(\cowt{})$,
$\mathcal{P}(\cowt{}+1) e_{21} = e_{21}\mathcal{P}(\cowt{})$. Let us write
$\alpha = \hbar\cowt{}-\hhalf$ and $\beta = \hhalf \sqrt{2C+1}$. Then,
\begin{align*}
\varphi(E)\varphi(F) &= \lp\frac{\hbar}{q-q^{-1}}\rp^2
G^+\lp \hbar\cowt{} - \hhalf(1+\sqrt{2C+1})\rp G^+\lp \hbar\cowt{} - \hhalf(1-\sqrt{2C+1})\rp \cdot \\
& \phantom{AAA}\cdot 
G^-\lp \hbar\cowt{} - \hhalf(1+\sqrt{2C+1})\rp G^-\lp \hbar\cowt{} - \hhalf(1-\sqrt{2C+1})\rp e_{12}e_{21} \\
&= \lp\frac{\hbar}{q-q^{-1}}\rp^2
\frac{e^{\alpha}+e^{-\alpha}-e^{\beta}-e^{-\beta}}{\alpha^2-\beta^2} e_{12}e_{21} \\
&= \frac{e^{\beta}+e^{-\beta}-e^{\alpha}-e^{-\alpha}}{(q-q^{-1})^2},
\end{align*}
where, the second equality follows from  $G^+(x)G^-(x) = (e^{x/2}-e^{-x/2})/x$ and
the last one from
\[
\hbar^2 e_{12}e_{21} = \hbar^2\lp\frac{C}{2} - \cowt{}^2 + \cowt{}\rp = \beta^2 - \alpha^2.
\]
Similarly, one gets
\[
\varphi(F)\varphi(E) = \frac{e^{\beta}+e^{-\beta}-e^{\gamma}-e^{-\gamma}}{(q-q^{-1})^2},
\]
where $\gamma = \hbar\cowt{} + \hhalf$.
Combining these, we obtain the desired identity as follows:
\begin{align*}
[\varphi(E),\varphi(F)] &= \frac{1}{(q-q^{-1})^2} \lp
e^{\gamma}+e^{-\gamma} - e^{\alpha}-e^{-\alpha} \rp \\
&= \frac{1}{(q-q^{-1})^2} \lp qK + q^{-1}K^{-1} - q^{-1}K - qK^{-1}\rp \\
&= \frac{K-K^{-1}}{q-q^{-1}}.
\end{align*}

\subsection{Evaluation homomorphism for $n=2$}\label{ssec:n=2ev}

Let us illustrate our main idea in the $n=2$ example. That is, we will now explain
how we obtained the expression of $\varphi$ in the previous subsection. For this,
we will have to write the algebra homomorphism $\ev : \Yhsl{2}\to \Usl{2}[\hbar]$
(see formulae \eqref{eq:evxi}, \eqref{eq:evx+} and \eqref{eq:evx-}):

\begin{align*}
\ev(\xi(u)) &= \frac{\Psl_2(u)}{\Psl_1\lp u+\frac{\hbar}{2}\rp \Psl_1\lp u-\frac{\hbar}{2}\rp}
 = \frac{u^2 - \lp \frac{\hbar}{2}\rp^2\lp 2C+1\rp}{\lp u-\frac{\hbar}{2}(h+1)\rp \lp u-\frac{\hbar}{2}(h-1)\rp}\ ,\\
\ev(x^+(u)) &= \lp u-\frac{\hbar}{2}(h-1)\rp^{-1}\cdot \hbar e_{12}\ , \\
\ev(x^-(u)) &= \lp u-\frac{\hbar}{2}(h+1)\rp^{-1}\cdot \hbar e_{21}\ . \\
\end{align*}

As explained in Corollary \ref{cor:main2}, if we write $\Psl_2(u)$ in its splitting
extension: 
$$
\Psl_2(u) = \lp u - \frac{\hbar}{2}\sqrt{2C+1}\rp \lp u + \frac{\hbar}{2} \sqrt{2C+1}\rp,
$$
we readily obtain the formulae for $\varphi(E)$ and $\varphi(F)$. This also highlights
the reason why the resulting composition is best expressed in terms of, while still independent of,
a choice of roots of the polynomials $\Psl_k(u)$.

\subsection{Proof of Theorem \ref{thm:main1}}\label{ssec:pf-main1}

The map $\varphi$ given in Section \ref{ssec:iso} is obtained via the following
composition, where $\Yhsl{n}$ is the Yangian of $\sl_n$ which is naturally an
$\N$--graded algebra, and $\wh{\Yhsl{n}}$ is its completion \wrt the $\N$--grading
(see Section \ref{ssec:ysln} below for the definition):
\[
\xymatrix@C=2cm{
\Uhsl{n} \ar[r]^{\Phi} \ar[dr]_{\varphi} & \wh{\Yhsl{n}} \ar[d]^{\ev}\\
& \hext{\Usl{n}}
}
\]
The expressions \eqref{eq:phiEk} and \eqref{eq:phiFk} are obtained by combining
Corollary \ref{cor:main2} with the explicit formulae for $\ev$ given in 
Proposition \ref{pr:partialfractions}. Thus the fact that $\varphi$ is an
algebra homomorphism follows from the corresponding assertions for $\Phi$
(proved in Theorem \ref{thm:main2}) and $\ev$ (Theorem \ref{thm:evishom}).\\

The reader can readily verify that modulo $\hbar$, $\varphi$ is the identity. Namely,
let $\ol{\varphi}$ be the induced map $\Uhsl{n}/\hbar\Uhsl{n} \to \Usl{n}$.
Then $\ol{\varphi}(E_k) = e_{k,k+1}$ and $\ol{\varphi}(F_k) = e_{k+1,k}$.
Since the quantum group $\Uhsl{n}$ is a flat deformation of $\Usl{n}$, this
implies that the algebra homomorphism $\varphi$ is in fact an isomorphism.


\section{The Yangian of $\sl_n$ and $\Uhsl{n}$}\label{sec:ysln}

In this section, we review the definition of the Yangian $\Yhsl{n}$, as given
in \cite{drinfeld-yangian-qaffine}. We also review the main construction
of \cite{sachin-valerio-1} yielding an algebra homomorphism between $\Uhsl{n}$
and the completion of $\Yhsl{n}$ \wrt its $\N$--grading.\\

\subsection{The Yangian of $\sl_n$}\label{ssec:ysln}

$\Yhsl{n}$ is a unital associative algebra over $\C[\hbar]$ generated
by $\{\xi_{i,r}, x_{i,r}^{\pm}\}_{r\in\N, i\in\bfI}$ subject to the following
relations
\begin{enumerate}
\item[(Y1)] For any $i,j\in\bfI$, $r,s\in\N$
\[[\xi_{i,r}, \xi_{j,s}] = 0.\]
\item[(Y2)] For $i,j\in\bfI$ and $s\in \N$
\[[\xi_{i,0}, x_{j,s}^{\pm}] = \pm a_{ij} x_{j,s}^{\pm}.\]
\item[(Y3)] For $i,j\in\bfI$ and $r,s\in\N$
\[[\xi_{i,r+1}, x^{\pm}_{j,s}] - [\xi_{i,r},x^{\pm}_{j,s+1}] =
\pm a_{ij}\hhalf (\xi_{i,r}x^{\pm}_{j,s} + x^{\pm}_{j,s}\xi_{i,r}).\]
\item[(Y4)] For $i,j\in\bfI$ and $r,s\in \N$
\[
[x^{\pm}_{i,r+1}, x^{\pm}_{j,s}] - [x^{\pm}_{i,r},x^{\pm}_{j,s+1}]=
\pm a_{ij}\hhalf (x^{\pm}_{i,r}x^{\pm}_{j,s} + x^{\pm}_{j,s}x^{\pm}_{i,r}).
\]
\item[(Y5)] For $i,j\in\bfI$ and $r,s\in \N$
\[[x^+_{i,r}, x^-_{j,s}] = \delta_{ij} \xi_{i,r+s}.\]
\item[(Y6)] Let $i\not= j\in\bfI$ and set $m = 1-a_{ij}$. For any
$r_1,\cdots, r_m\in \N$ and $s\in \N$
\[\sum_{\pi\in\mathfrak{S}_m}
\left[x^{\pm}_{i,r_{\pi(1)}},\left[x^{\pm}_{i,r_{\pi(2)}},\left[\cdots,
\left[x^{\pm}_{i,r_{\pi(m)}},x^{\pm}_{j,s}\right]\cdots\right]\right]\right.=0.\]
\end{enumerate}

Note that $\Yhsl{n}$ is a graded algebra, with $\deg(\hbar)=1$ and $\deg(y_{i,r}) = r$
for $y = \xi, x^{\pm}$.

\subsection{Formal currents}\label{ssec:yfields}

Define $\xi_i(u),x^\pm_i(u)\in
\Yhg[[u^{-1}]]$ by
\[\xi_i(u)=1 + \hbar\sum_{r\geqslant 0} \xi_{i,r}u^{-r-1}\aand
x^{\pm}_i(u)=\hbar\sum_{r\geqslant 0} x_{i,r}^{\pm} u^{-r-1}.\]

According to \cite[Prop.~2.3]{sachin-valerio-2}, the relations (Y1)--(Y5) are 
then equivalent to the following identities in $\Yhsl{n}[u,v;u^{-1},v^{-1}]]$.
\begin{enumerate}
\item[(\Y1)] For any $i,j\in\bfI$
\[[\xi_i(u), \xi_j(v)]=0.\]
\item[(\Y2)] For any $i,j\in \bfI$, and $a = \hbar a_{ij}/2$
\[(u-v\mp a)\xi_i(u)x_j^{\pm}(v)=
(u-v\pm a)x_j^{\pm}(v)\xi_i(u)\mp 2a x_j^{\pm}(u\mp a)\xi_i(u).\]
\item[(\Y3)] For any $i,j\in \bfI$, and $a = \hbar a_{ij}/2$
\begin{multline*}
(u-v\mp a) x_i^{\pm}(u)x_j^{\pm}(v)\\
= (u-v\pm a)x_j^{\pm}(v)x_i^{\pm}(u)
+\hbar\lp [x_{i,0}^{\pm},x_j^{\pm}(v)] - [x_i^{\pm}(u),x_{j,0}^{\pm}]\rp.
\end{multline*}
\item[(\Y4)] For any $i,j\in \bfI$
\[(u-v)[x_i^+(u),x_j^-(v)]=-\delta_{ij}\hbar\left(\xi_i(u)-\xi_i(v)\right).\]
\end{enumerate}

We also recall that the relation (Y6) follows from (Y1)--(Y5) and
the special case of (Y6) when all $r_1=\cdots=r_m=s=0$
\cite{levendorskii}. 

\begin{lem}\label{lem:Y2prime}
The relation (\Y2) is equivalent to the following
\begin{equation}\label{eq:Y2prime}
\tag{\Y 2$'$}
\Ad(\xi_i(u))^{-1}(x_j^{\pm}(v)) = \frac{u-v\mp a}{u-v\pm a}x_j^{\pm}(v)
\pm \frac{2a}{u-v\pm a} x_j^{\pm}(u\pm a),
\end{equation}
where as before $a = a_{ij}\hbar/2$.
\end{lem}
\begin{pf}
Setting $v=u\pm a$ in (\Y2) we obtain
\[
\Ad(\xi_i(u))x_j^{\pm}(u\pm a) = x_j^{\pm}(u\mp a).
\]
Note that this relation can be similarly obtained from \eqref{eq:Y2prime}.
Using this identity we can deduce (\Y2) from \eqref{eq:Y2prime} and
vice versa.
\end{pf}

\subsection{Some elementary equivalences among relations of $\Yhsl{n}$}\label{ssec:reduction}

The following proposition will be used to reduce the list of relations
to be verified in order to obtain an algebra homomorphism from
$\Yhsl{n}$ to $\Usl{n}[\hbar]$.

\begin{prop}\label{pr:reduction}\hfill
\begin{enumerate}
\item 
Assuming the relation ($\mathcal{Y}$1), (\Y2) follows from the following
\begin{itemize}
\item The $i=j$ case of (\Y2)
(or, equivalently \eqref{eq:Y2prime}).
\item For $i\neq j$, either the following special case of (\Y2):
\[
\Ad(\xi_i(u))(x_{j,0}^{\pm}) = x_{j,0}^{\pm} \pm a_{ij} x_j^{\pm}(u\mp a_{ij}\hbar/2),
\]
or the analogous special case of \eqref{eq:Y2prime}:
\[
\Ad(\xi_i(u))^{-1}(x_{j,0}^{\pm}) = x_{j,0}^{\pm} \mp a_{ij} x_j^{\pm}(u\pm a_{ij}\hbar/2).
\]
\end{itemize}

\item Assuming (\Y1) and (\Y2), the relation (\Y3) follows from
\begin{itemize}
\item The following special case of (\Y3), for each $i,j\in\bfI$ such 
that $i=j$, or $a_{ij}=0$.
\[
[x_{i,0}^{\pm},x_j^{\pm}(u)] - [x_i^{\pm}(u),x_{j,0}^{\pm}] =
\mp \frac{a_{ij}}{2} (x_i^{\pm}(u)x_j^{\pm}(u) + x_j^{\pm}(u)x_i^{\pm}(u)).
\]
\item The relation (\Y3) for $j=i+1$.
\end{itemize}

\item Again assuming (\Y1) and (\Y2), the relation (\Y4) follows from its special
case: for each $i,j\in\bfI$
\[
[x_i^+(u),x_{j,0}^-] = \delta_{ij}(\xi_i(u)-1).
\]
\end{enumerate}
\end{prop}

\begin{pf}
We begin by proving (\Y2) assuming its special cases listed in (1) above hold. 
Let $i,j,k\in\bfI$ and assume
that we know the following relations from (S3)
\begin{align*}
\Ad(\xi_i(u))(x_{k,0}^{\pm}) &= x_{k,0}^{\pm} \pm a_{ik} x_k^{\pm}(u\mp a_{ik}\hbar/2),\\
\Ad(\xi_j(u))(x_{k,0}^{\pm}) &= x_{k,0}^{\pm} \pm a_{jk} x_k^{\pm}(u\mp a_{jk}\hbar/2).
\end{align*}
Now we compute $\Ad(\xi_i(u)(\xi_j(v))x_{k,0}^{\pm}$ in two different ways,
since we know $\xi_i(u)$ and $\xi_j(v)$ commute, from (\Y1).
\begin{align*}
\Ad(\xi_i(u)\xi_j(v))(x_{k,0}^{\pm}) &=
\Ad(\xi_i(u))\lp x_{k,0}^{\pm} \pm a_{jk}x_k^{\pm}(v\mp a_{jk}\hbar/2)\rp \\
&= x_{k,0}^{\pm}\pm a_{ik} x_k^{\pm}(u\mp a_{ik}\hbar/2)
\pm a_{jk}\Ad(\xi_i(u))(x_k^{\pm}(v\mp a_{jk}\hbar/2)).
\end{align*}
Similarly we get
\begin{align*}
\Ad(\xi_j(v)\xi_i(u))(x_{k,0}^{\pm}) &=
\Ad(\xi_j(v))\lp x_{k,0}^{\pm} \pm a_{ik}x_k^{\pm}(u\mp a_{ik}\hbar/2)\rp \\
&= x_{k,0}^{\pm}\pm a_{jk} x_k^{\pm}(v\mp a_{jk}\hbar/2)
\pm a_{ik}\Ad(\xi_j(v))(x_k^{\pm}(u\mp a_{ik}\hbar/2)).
\end{align*}
Combining we obtain the following equation
\[
a_{jk}((\Ad(\xi_i(u))-1)(x_k^{\pm}(v\mp a_{jk}\hbar/2))) = 
a_{ik}((\Ad(\xi_j(v))-1)(x_k^{\pm}(u\mp a_{ik}\hbar/2))).
\]
The conclusion is that if we know $\Ad(\xi_i(u))(x_k^{\pm}(v))$ for some $i\in\bfI$
so that $a_{ik}\neq 0$, then we can compute $\Ad(\xi_j(u))(x_k^{\pm}(v))$ for any $j\in\bfI$.
(1) asserts exactly that we know (\Y2) for one such pair and we are done.\\

The proof of the remaining relations uses (\Y2) which will be assumed. 
For instance, let us prove (\Y4) from its special cases given in (3). 
The proof of (2) is entirely analogous
and is skipped here.\\

Apply $\Ad(\xi_j(v))$ to both sides of 
\[
[x_i^+(u),x_{j,0}^-] = \delta_{ij} (\xi_i(u)-1).
\] 
Using (\Y1), the right--hand
side does not change, while the left hand side can be computed as follows (where
$a=a_{ij}\hbar/2$):
\begin{align*}
\Ad(\xi_j(v))([x_i^+(u),x_{j,0}^-]) &= \left[
\frac{v-u+a}{v-u-a}x_i^+(u) - \frac{2a}{v-u-a}x_i^+(v-a),
x_{j,0}^- - 2x_j^-(v+\hbar)\right].
\end{align*}

Now, for $i\neq j$ we get
\[
(v-u+a)[x_i^+(u), x_j^-(v+\hbar)] = 2a [x_i^+(v-a),x_j^-(v+\hbar)].
\]
Set $u=v+a$ in the equation above to see that its \rhs must be zero. Thus
so must be its \lhs and we obtain (\Y4).\\

Assuming $i=j$, we can drop the subscript $i$ and note that $a = \hbar$. We have
\begin{align*}
\Ad(\xi(v))([x^+(u),x_{0}^-]) &= 
\frac{v-u+\hbar}{v-u-\hbar}(\xi(u)-1) - \frac{2\hbar}{v-u-\hbar}(\xi(v-\hbar)-1) \\
&\phantom{=}
-2\frac{v-u+\hbar}{v-u-\hbar}[x^+(u),x^-(v+\hbar)]
+\frac{4\hbar}{v-u-\hbar}[x^+(v-\hbar),x^-(v+\hbar)].
\end{align*}
Setting this equal to $\xi(u)-1$ we get the following equation, after clearing
the denominator and cancelling a factor of $2$:
\begin{equation}\label{eq:pfred1}
(u-v-\hbar)[x^+(u),x^-(v+\hbar)]+2\hbar[x^+(v-\hbar),x^-(v+\hbar)]
= \hbar(\xi(v-\hbar)-\xi(u)).
\end{equation}
Set $u=v+\hbar$ to get $2\hbar[x^+(v-\hbar),x^-(v+\hbar)]=\hbar(\xi(v-\hbar)-\xi(v+\hbar))$.

Now replace the commutator $[x^+(v-\hbar),x^-(v+\hbar)]$
in \eqref{eq:pfred1} by this to get
\[
(u-v-\hbar)[x^+(u),x^-(v+\hbar)] = \hbar(\xi(v+\hbar)-\xi(u)),
\]
which is exactly (\Y4) for $i=j$.
\end{pf}

\subsection{Homomorphism $\Phi : \Uhsl{n}\to\wh{\Yhsl{n}}$}\label{ssec:homPhi}

Now let $\wh{\Yhsl{n}}$ be the completion of $\Yhsl{n}$ \wrt its $\N$--grading.
Again let $G^{\pm}(x)$ be two formal series in $1+x\C[\negthinspace[x]\negthinspace]$ satisfying
\eqref{eq:factorizationproblem}. \\

Following \cite[\S 2.9]{sachin-valerio-1}, we define for each $i\in\bfI$:
\begin{equation}\label{eq:Bti}
\begin{split}
t_i(u) &= \hbar\sum_{r\in\N} t_{i,r}u^{-r-1} := \log(\xi_i(u)),\\
B_i(v) &= \hbar\sum_{r\in\N} t_{i,r} \frac{v^r}{r!}.
\end{split}
\end{equation}

Let $Y^0$ be the subalgebra of $\Yhsl{n}$ generated by $\{\xi_{i,r}\}_{i\in\bfI,r\in\N}$.
Define $g_i^{\pm}(u) = \sum_{m\in\N} g^{\pm}_{i,m}u^m \in \wh{Y^0[u]}$ by 
\begin{equation}\label{eq:gpm}
g_i^{\pm}(u) := \frac{1}{G^+(\hbar)}\exp\lp B_i(-\partial_v)\cdot
\frac{d}{dv} \lp \log(G^{\pm}(v))\rp \rp.
\end{equation}

\begin{thm}\label{thm:main2}
The assignment $\Phi(H_i) = \xi_{i,0}$ and 
\[
\Phi(E_i) = \sum_{m\in\N} g_{i,m}^+ x_{i,m}^+ \aand
\Phi(F_i) = \sum_{m\in\N} g_{i,m}^- x_{i,m}^-
\]
defines an algebra homomorphism $\Phi : \Uhsl{n}\to \wh{\Yhsl{n}}$.
\end{thm}

\begin{pf}
In \cite[Prop.~2.10]{sachin-valerio-1} certain algebra homomorphisms
$\lambda_i^{\pm}(u) = \sum_{r\in\N} \lambda^{\pm}_{i;r} u^r : Y^0
\to Y^0[u]$ are constructed so that
\begin{equation}\label{eq:pf-main2-1}
\lambda_i^{\pm}(v_1)B_j(v_2) = B_j(v_2) \mp \frac{e^{\hhalf a_{ij} v_2}
-e^{-\hhalf a_{ij}v_2}}{v_2} e^{v_1v_2}.
\end{equation}
Note that we are in a simply--laced case, so we don't need to introduce the symmetrizing
integers present in (7) of \cite[Prop.~2.10]{sachin-valerio-1}. The necessary and sufficient
conditions prescribed in \cite[Thm.~3.4, \S 4.7]{sachin-valerio-1} for $\Phi$ to be an
algebra homomorphism are:
\begin{itemize}
\item[(A)] For each $i,j\in\bfI$
\[
g_i^+(u)\lambda_i^+(u)(g_j^-(v)) = g_j^-(v)\lambda_j^-(v)(g_i^+(u)).
\]
\item[($\wt{B}$)] For every $i\in\bfI$, we have
\[
g_i^+(v)\lambda_i^+(v)(g_i^-(v)) = \frac{\hbar}{q-q^{-1}}\exp\lp
B_i(-\partial_v)\partial_v \cdot
\log\lp \frac{e^{v/2}-e^{-v/2}}{v}\rp \rp.
\]
\item[(C)] For each $i,j\in\bfI$, we have
\[
g_i^{\pm}(u)\lambda_i^{\pm}(u)(g_j^{\pm}(v))\frac{e^u-e^{v\pm a}}{u-v\mp a} = 
g_j^{\pm}(v)\lambda_j^{\pm}(v)(g_i^{\pm}(u))\frac{e^{u\pm a}-e^v}{u-v\pm a}. 
\]
\end{itemize}

Thus we need to compute $\lambda_i^{\epsilon_1}(u)(g_j^{\epsilon_2}(v))$ for each 
$i,j\in\bfI$ and $\epsilon_1,\epsilon_2\in\{\pm\}$. For this we have the following\\

\noindent {\bf Claim.} Let $a = \hhalf a_{ij}$. Then we have 
\[ 
\lambda_i^{\epsilon_1}(u)(g_j^{\epsilon_2}(v)) = 
g_j^{\epsilon_2}(v) \lp\frac{G^{\epsilon_2}(v-u-a)}{G^{\epsilon_2}(v-u+a)}\rp^{\epsilon_1 1}.
\]

Given the claim we can prove that the equations (A), ($\wt{B}$) and (C) hold, as follows.\\

\noindent {\em Proof of (A).} This equation becomes
\[
\frac{G^-(v-u-a)}{G^-(v-u+a)} = \frac{G^+(u-v-a)}{G^+(u-v+a)},
\]
which is true since $G^-(x) = G^+(-x)$ as per \eqref{eq:factorizationproblem}.\\

\noindent {\em Proof of ($\wt{B}$).} The \lhs of condition ($\wt{B}$) can be computed
using the claim above:
\begin{align*}
g_i^+(v)\lambda_i^+(v)(g_i^-(v)) &= g_i^+(v)g_i^-(v)\frac{G^+(\hbar)}{G^-(\hbar)} \\
&= \frac{1}{G^+(\hbar)G^-(\hbar)}\exp\lp B_i(-\partial_v)\partial_v\cdot
\log\lp G^+(x)G^-(x)\rp \rp \\
&= \frac{\hbar}{q-q^{-1}} \exp\lp B_i(-\partial_v)\partial_v\cdot
\log\lp \frac{e^{v/2}-e^{-v/2}}{v}\rp \rp,
\end{align*}
where we used that $G^+(x)G^-(x) = (e^{x/2}-e^{-x/2})/v$ as required in
\eqref{eq:factorizationproblem}.\\

\noindent {\em Proof of (C).} This condition (for the $+$ case) takes the following form:
\[
\frac{G^+(v-u-a)}{G^+(v-u+a)} \frac{e^u - e^{v+a}}{u-v-a}
= \frac{G^+(u-v-a)}{G^+(u-v+a)} \frac{e^{u+a} - e^v}{u-v+a},
\]
which again follows from \eqref{eq:factorizationproblem}.\\

It remains to prove the claim above. Let us take $\epsilon_1 = +$ and $\epsilon_2 = -$
for definiteness, and as usual let $a = \hhalf a_{ij}$. Then we get
\begin{align*}
\lambda_i^+(u)(g_j^-(v)) &= G^+(\hbar)^{-1} \exp \lp
\lp
B_j(-\partial_v)
- \frac{e^{-a\partial_v}-e^{a\partial_v}}{(-\partial_v)}
e^{-u\partial_v}
\rp\cdot
\partial_v \log(G^-(v))
\rp\\
&= g_j^-(v) \exp\lp \lp e^{-\partial_v(u+a)} - e^{-\partial_v(u-a)}\rp
\cdot \log(G^-(v))\rp \\
&= g_j^-(v)\frac{G^-(v-u-a)}{G^-(v-u+a)}
\end{align*}
as claimed.
\end{pf}

\subsection{Composition of a graded algebra homomorphism
with $\Phi$}\label{ssec:cor-main2}

Let us fix $i\in\bfI$ and consider the following situation.
Assume $A$ is an $\N$--graded, unital algebra over $\C[\hbar]$, and assume that we are given a homomorphism
of graded algebras
$\eta : \Yhsl{n} \to A$ such that
\begin{itemize}
\item $\eta(\xi_i(u))$ is expansion in $u^{-1}$ of a rational function of the form:
\[
\eta(\xi_i(u)) = \prod_{k=1}^N \frac{u-a_k}{u-b_k},
\]
where $a_k,b_k\in A$ are homogeneous elements of degree $1$, for $1\leqslant k\leqslant N$.

\item $\eta(x_i^{\pm}(u))$ are again rational functions of the form:
\[
\eta(x_i^{\pm}(u)) = \sum_{\ell=1}^M \frac{\hbar}{u-c^{\pm}_{\ell}} B_{\ell}^{\pm},
\]
where $c^{\pm}_{\ell} \in A$ are of degree $1$ and $B_{\ell}^{\pm}\in A$ are of degree 0.
\end{itemize}

\begin{cor}\label{cor:main2}
The composition $\eta\circ \Phi : \Uhsl{n} \to \wh{A}$ maps $E_i, F_i$ to the following:
\begin{align*}
E_i &\mapsto \frac{1}{G^+(\hbar)} \sum_{\ell=1}^M \lp\prod_{k=1}^N 
\frac{G^+(c^+_{\ell}-a_k)}{G^+(c^+_{\ell}-b_k)}\rp B_{\ell}^+,\\
F_i &\mapsto \frac{1}{G^+(\hbar)} \sum_{\ell=1}^M \lp\prod_{k=1}^N 
\frac{G^-(c^-_{\ell}-a_k)}{G^-(c^-_{\ell}-b_k)}\rp B_{\ell}^-.
\end{align*}
\end{cor}

\begin{pf}
The proof follows a computation similar to the one given in \cite[Section 4.6]{sachin-valerio-1}. Since
$\ds \eta(\xi_i(u)) = \prod_{k=1}^N \frac{u-a_k}{u-b_k}$, we get
\[
\eta(t_i(u)) = \sum_{k=1}^N \log(1-a_ku^{-1}) - \log(1-b_ku^{-1}) = \sum_{k=1}^N 
\lp \sum_{r\geqslant 0} \frac{b_k^{r+1}-a_k^{r+1}}{r+1} u^{-r-1}\rp.
\]
Thus $\hbar\eta(t_{i,r}) = \ds \sum_{k=1}^N \frac{b_k^{r+1}-a_k^{r+1}}{r+1}$. This implies
\begin{align*}
\eta(g_i^{\pm}(u)) &= G^+(\hbar)^{-1} \exp\lp
\sum_{r\geqslant 0} (-1)^r \lp \sum_{k=1}^N \frac{b_k^{r+1}-a_k^{r+1}}{(r+1)!} \rp \partial_u^{r+1} \log(G^{\pm}(u))
\rp\\
&= G^+(\hbar)^{-1} \exp\lp \lp\sum_{k=1}^N e^{-a_k\partial_u} - e^{-b_k\partial_u}\rp \log(G^{\pm}(u))\rp \\
&= G^+(\hbar)^{-1} \exp\lp \sum_{k=1}^N \log(G^{\pm}(u-a_k)) - \log(G^{\pm}(u-b_k))\rp \\
&= G^+(\hbar)^{-1} \prod_{k=1}^N \frac{G^{\pm}(u-a_k)}{G^{\pm}(u-b_k)}.
\end{align*}
Finally from the expression of $\eta(x_i^{\pm}(u)$ we get that 
$\eta(x_{i,m}^{\pm}) = \ds \sum_{\ell=1}^M (c_{\ell}^{\pm})^m B_{\ell}^{\pm}$.
Substituting this in the formula for $\Phi(E_i)$ and $\Phi(F_i)$ 
given in Theorem \ref{thm:main2}, we get
\begin{align*}
\sum_{m\geqslant 0} g_{i,m}^{\pm}x_{i,m}^{\pm} &= \sum_{\ell=1}^M \lp\sum_{m\geqslant 0} g_{i,m}^{\pm}(c_{\ell}^{\pm})^m\rp B_{\ell}^{\pm} 
= \sum_{\ell=1}^M g_i^{\pm}(c_{\ell}^{\pm}) B_{\ell}^{\pm} \\
&= G^+(\hbar)^{-1} \sum_{\ell=1}^M \lp \prod_{k=1}^N \frac{G^{\pm}(c_\ell^\pm-a_k)}{G^{\pm}(c_\ell^\pm-b_k)}\rp B_\ell^\pm
\end{align*}
as claimed.
\end{pf}


\section{$\RTT$ relations and determinant identities}\label{sec:qminsl}

In this section, we study the algebraic properties of the matrix $\Tsl(u)$.
We show that it satisfies the $RTT$ relations and obtain commutation
relations between quantum minors of $\Tsl(u)$. In particular, we prove
the \emph{Capelli identity}
for $\sl_n$, \ie the coefficients of the quantum--determinant of $\Tsl(u)$ 
generate the center of $U\sl_n$.

\subsection{$RTT$ relations}\label{ssec:RTT-sl}

Let $\Tsl(u)$ be the $n\times n$ matrix with coefficients
from $\Usl{n}[\hbar,u]$ as defined in Section \ref{ssec:transfer}.

We view this matrix as an element of $\aux{n}\ten\Usl{n}[u,\hbar]$ as follows.
Let $\{\bs{i}\}_{1\leqslant i\leqslant n}$ be the standard basis of $\C^n$ and let
$\elm{i}{j}$ be the elementary matrix defined as:
$\elm{i}{j}\bs{k} = \delta_{jk}\bs{i}$. Then
\begin{equation}
\Tsl(u)=\sum_{i,j}\elm{i}{j}\ten \Tsl_{ij}(u).
\end{equation}

Thus, we have $\Tsl(u)\bs{j}=\sum_i \bs{i}\ten \Tsl_{ij}(u)$. Let $P\in\auxaux{n}$ 
be the flip of the tensor factors, and let $R(u)=u\Id + \hbar P$ be the Yang's $R$--matrix.
The following (Yang--Baxter) equation holds in $\haux{n}{3}[\hbar,u,v]$:
\begin{equation}\label{eq:ybe}
\tag{YBE}
R_{12}(u)R_{13}(u+v)R_{23}(v) = R_{23}(v)R_{13}(u+v)R_{12}(u),
\end{equation}
where, as usual, the subscripts indicate which tensor factors $R$ acts on.

\begin{prop}\label{pr:rtt}
Set 
\[\Tsl_1(u)=\sum_{i,j} \elm{i}{j}\ten 1\ten \Tsl_{ij}(u)
\aand \Tsl_2(v)=\sum_{i,j}1\ten \elm{i}{j}\ten \Tsl_{ij}(v)
\]
in $\auxaux{n}\ten\Usl{n}[\hbar,u,v]$. Then,
\begin{equation}\label{eq:RXX}
R(u-v)\Tsl_1(u)\Tsl_2(v)=\Tsl_2(v)\Tsl_1(u)R(u-v).
\end{equation}
\end{prop}

\begin{pf}
We apply both sides of \eqref{eq:RXX} to an arbitrary basis vector $\bs{j}\ten \bs{l}\in\C^n\ten\C^n$.
For the left--hand side, we get
\[
(u-v)
\sum_{i,k}\bs{i}\ten \bs{k}\ten \Tsl_{ij}(u)\Tsl_{kl}(v)
+
\hbar\sum_{i,k}\bs{k}\ten \bs{i}\ten \Tsl_{ij}(u)\Tsl_{kl}(v),
\]
while the right--hand side gives
\[
(u-v)
\sum_{i,k}\bs{i}\ten \bs{k}\ten \Tsl_{kl}(v)\Tsl_{ij}(u)
+
\hbar\sum_{i,k}\bs{i}\ten \bs{k}\ten \Tsl_{kj}(v)\Tsl_{il}(u).
\]
Thus, we have to prove the following equation for each $i,j,k,l$:
\begin{equation}\label{eq:Xu-entry-rel}
(u-v)[\Tsl_{ij}(u), \Tsl_{kl}(v)]=\hbar
\left(\Tsl_{kj}(v)\Tsl_{il}(u)-\Tsl_{kj}(u)\Tsl_{il}(v)\right).
\end{equation}
Switching the roles of $u\leftrightarrow v$, $ij \leftrightarrow kl$,
the equation above is equivalent to
\begin{equation}\label{eq:Xu-entry-rel-eq}
(u-v)[\Tsl_{ij}(u), \Tsl_{kl}(v)]=\hbar
\left(\Tsl_{il}(u)\Tsl_{kj}(v)-\Tsl_{il}(v)\Tsl_{kj}(u)\right).
\end{equation}
Note that the entries of the matrix $\Tsl$ satisfy the following relation
\begin{equation}\label{eq:X-entry-rel}
[\Tsl_{ij}, \Tsl_{kl}]=\delta_{jk}\Tsl_{il}-\delta_{il}\Tsl_{kj}.
\end{equation}
From this it is easy to deduce \eqref{eq:Xu-entry-rel} as follows:
\begin{align*}
(u-v)[\Tsl_{ij}(u), \Tsl_{kl}(v)] &= 
 \hbar^2(u-v)[\Tsl_{ij}, \Tsl_{kl}] \\
&=\hbar^2(u-v)(\delta_{kj}\Tsl_{il}-\delta_{il}\Tsl_{kj})\\
&= \hbar((\delta_{il}u-\hbar \Tsl_{il})(\delta_{kj}v-\hbar \Tsl_{kj})- \\
& \phantom{=}(\delta_{il}v-\hbar \Tsl_{il})(\delta_{kj}u-\hbar \Tsl_{kj}))\\
&= \hbar\left(\Tsl_{il}(u)\Tsl_{kj}(v)-\Tsl_{il}(v)\Tsl_{kj}(u)\right).
\end{align*}
This finishes the proof of the proposition.
\end{pf}

\subsection{Several variables generalization}\label{ssec:higher-RTT}

For $N\geqslant 2$, consider the following element of $\haux{n}{N}$, depending on $u_1,\cdots,u_N$:
\begin{multline*}
R(u_1, \dots, u_N) := R_{N-1,N}\cdot(R_{N-2,N}R_{N-2,N-1})\cdots(R_{1N}\cdots R_{12})=
\\
= (R_{12}\cdots R_{1N})\cdot (R_{N-2,N-1}R_{N-2,N})\cdots R_{N-1,N}
\end{multline*}
where $R_{ij} = R_{ij}(u_i-u_j)$ acts on $i^{\scriptstyle{\text{th}}}$ and
$j^{\scriptstyle{\text{th}}}$ tensor factor. The equality of the two expressions
given above follows by a repeated application of the Yang--Baxter equation
\eqref{eq:ybe} (see also the proof of the following proposition).

\begin{prop} The matrix $\Tsl(u)$ satisfies
\begin{equation}\label{eq:gen-RXX}
R(u_1,\dots, u_N)\Tsl_1(u_1)\cdots \Tsl_N(u_N)=\Tsl_N(u_N)\cdots \Tsl_1(u_1)R(u_1, \dots, u_N).
\end{equation}
\end{prop}

\begin{pf}
We proceed by induction. For $N=2$, one has \eqref{eq:RXX}.
For $N>2$, one has
\begin{align*}
(R_{1N}\cdots R_{13})R_{12}\Tsl_1\Tsl_2&(\Tsl_3\cdots \Tsl_N)=\\
=&(R_{1N}\cdots R_{13})\Tsl_2\Tsl_1R_{12}(\Tsl_3\cdots \Tsl_N)=\\
=&\Tsl_2(R_{1N}\cdots R_{14})R_{13}\Tsl_1\Tsl_3(\Tsl_4\cdots \Tsl_N)R_{12}=\\
=&(\Tsl_2\cdots \Tsl_N)\Tsl_1(R_{1N}\cdots R_{12}).
\end{align*}
Since $R(u_1,\dots, u_N)=R(u_2,\dots, u_N)(R_{1N}\cdots R_{12})$, we get
\begin{align*}
R(u_1,\dots, u_N)&\Tsl_1\cdots \Tsl_N=\\
=&R(u_2,\dots, u_N)(R_{1N}\cdots R_{12})\Tsl_1\cdots \Tsl_N=\\
=&R(u_2,\dots, u_N)(\Tsl_2\cdots \Tsl_N)\Tsl_1(R_{1N}\cdots R_{12})
\end{align*}
and the result follows by induction.
\end{pf}

\subsection{Specialization}\label{ssec:RTT-antisym}

Let $A_N$ be the antisymmetriser operator $A_N=\sum_{\sigma\in\mathfrak{S}_N}(-1)^{\sigma}\sigma$.

\begin{prop}
If $u_i-u_{i+1}=-\hbar$ for all $i=1,\dots, N-1$, then
\begin{equation*}
R(u_1,\dots, u_N) = c_N A_N,
\end{equation*}
where $c_N \in \C[\hbar]$ is a scalar. Explicitly, $c_N = (-\hbar)^{N(N-1)/2} (N-1)! \cdots 1!$.
\end{prop}

\begin{pf}
For $N=2$, $R(-\hbar)=(-\hbar)(1-P) = -\hbar A_2$. For $N>2$, one has
\[
R(u_1,\dots, u_N)=c_{N-1}\wt{A}_{N-1}(R_{1N}\cdots R_{12}),
\]
where $\wt{A}_{N-1}$ is the antisymmetriser operator on $\{2,\dots, N\}$. Then
\begin{align*}
R(u_1,\dots, u_N)=&c_{N-1}\wt{A}_{N-1}(R_{1N}\cdots R_{12})=\\
=&c_{N-1} (-\hbar)^{N-1}(N-1)!  \wt{A}_{N-1}\left(1-\frac{1}{N-1}P_{1N}\right)\cdots \left(1-P_{12}\right)=\\
=&c_N\wt{A}_{N-1}(1-P_{12}-\cdots -P_{1N})=c_N{A}_N
\end{align*}
as desired.
\end{pf}

For future reference, we will write the equation given by the proposition above
as $R(u_1,\ldots,u_N) \sim A_N$.

\begin{cor}\label{cor:antisym}
Set $u_i=u+\hbar(i-1)$. Then
\[
A_N\Tsl_1(u_1)\cdots \Tsl_N(u_N)=\Tsl_N(u_N)\cdots \Tsl_1(u_1)A_N.
\]
\end{cor}

\subsection{Quantum determinants}\label{ssec:qdet}

The quantum determinant of the matrix $\Tsl(u)$ is the element $\qdet(\Tsl(u))$
defined by the relation
\begin{equation}\label{eq:qdet}
A_n\qdet(\Tsl(u))=A_n\Tsl_1(u_1)\cdots \Tsl_n(u_n),
\end{equation}
where $u_i=u+\hbar(i-1)$ for $i=1,\dots, n$.

\begin{prop}
For every $\mu\in\fkS_n$,
\[
\qdet(\Tsl(u))=(-1)^{\mu}\sum_{\sigma\in\fkS_n}(-1)^{\sigma}\Tsl_{\sigma(1),\mu(1)}(u_1)
\cdots \Tsl_{\sigma(n),\mu(n)}(u_n).
\]
In particular,
\[
\qdet(\Tsl(u))=\sum_{\sigma\in\fkS_n}(-1)^{\sigma}\Tsl_{\sigma(1),1}(u_1)\cdots \Tsl_{\sigma(n),n}(u_n).
\]
\end{prop}

\begin{proof}
It is enough to apply both sides of \eqref{eq:qdet} to the vector
$\bs{\mu(1)}\ten\cdots\ten \bs{\mu(n)}$ in $(\C^{n})^{\ten n}$.
\end{proof}

Similarly, from Corollary \ref{cor:antisym} with $N=n$, one has the relation
\[
A_n\qdet(\Tsl(u))=\Tsl_n(u_n)\cdots \Tsl_1(u_1)A_n,
\]
providing a description of $\qdet(\Tsl(u))$ as a row--determinant.

\subsection{Quantum minors}\label{ssec:q-minors}

The quantum minors of $\Tsl(u)$ are also defined by the relation derived
in Corollary \ref{cor:antisym}. Let $N\leqslant n$ and let 
$\ul{a},\ul{b}\in\{1,\dots, n\}^N$. For convenience, we write 
$\bs{\ul{a}}$ for the basis vector $\bs{a_1}\otimes\cdots\otimes\bs{a_N}$.
Let $\mathcal{A}_N(u)$ be the operator given in Corollary \ref{cor:antisym}.
Then we define $\qmin{\ul{a}}{\ul{b}}{u}$ as the following matrix coefficient
of $\mathcal{A}_N(u)$:
\begin{equation}\label{eq:q-minors}
\mathcal{A}_N\bs{\ul{b}}=\sum_{\ul{a}} \bs{\ul{a}} \otimes \qmin{\ul{a}}{\ul{b}}{u}.
\end{equation}

The following is an obvious generalisation of \ref{ssec:qdet}.
For any $\ul{a}\in\{1,\dots, n\}^N$, we denote by $\ul{a}\setminus a_i$ the tuple obtained from
$\ul{a}$ by removing the $i^{\scriptstyle{\text{th}}}$ entry $a_i$.

\begin{lem}\label{lem:q-minors}
Let $u_j = u+\hbar(j-1)$ as before. Then we have the following:
\begin{enumerate}
\item For any tuples $\ul{a}=(a_1,\dots, a_N)$, $\ul{b}=(b_1,\dots, b_N)$, 
\begin{align*}
\qmin{\ul{a}}{\ul{b}}{u}
=&\sum_{\sigma\in\fkS_N}(-1)^{\sigma}\Tsl_{a_{\sigma(1)}, b_1}(u_1)\cdots \Tsl_{a_{\sigma(N)}, b_N}(u_N)\\
=&\sum_{\sigma\in\fkS_N}(-1)^{\sigma}\Tsl_{a_1, b_{\sigma(1)}}(u_N)\cdots \Tsl_{a_N, b_{\sigma(N)}}(u_1).
\end{align*}
\item For every $\sigma\in\fkS_N$,
\begin{align*}
\qmin{\sigma(\ul{a})}{\ul{b}}{u}=(-1)^{\sigma}\qmin{\ul{a}}{\ul{b}}{u}
=\qmin{\ul{a}}{\sigma(\ul{b})}{u}.
\end{align*}
\item
\begin{align*}
\qmin{\ul{a}}{\ul{b}}{u}
=&\sum_{k=1}^N(-1)^{N-k}\qmin{\ul{a}\setminus{a_k}}{\ul{b}\setminus b_{N}}{u}\cdot \Tsl_{a_kb_N}(u+\hbar(N-1))\\
=&\sum_{k=1}^N(-1)^{N-k}\qmin{\ul{a}\setminus{a_N}}{\ul{b}\setminus b_{k}}{u+\hbar}\cdot \Tsl_{a_Nb_k}(u)\\
=&\sum_{k=1}^N(-1)^{k-1}\Tsl_{a_kb_1}(u)\cdot \qmin{\ul{a}\setminus{a_k}}{\ul{b}\setminus b_{1}}{u+\hbar}\\
=&\sum_{k=1}^N(-1)^{k-1}
\Tsl_{a_1b_k}(u+\hbar(N-1))\cdot \qmin{\ul{a}\setminus{a_1}}{\ul{b}\setminus b_{k}}{u}.
\end{align*}
\end{enumerate}
\end{lem}

\subsection{Commutation relations with quantum minors}\label{ssec:qmin-comm}

For any $\ul{a}\in\{1,\dots, n\}^N$, we denote by $\rho_{i,x}(\ul{a})$ the tuple obtained from
$\ul{a}$ by replacing the $i^{\scriptstyle{\text{th}}}$ entry $a_i$ with $x$.

\begin{prop}\label{pr:minor-comm} 
For every $1\leqslant k,l\leqslant n$ and $\ul{a},\ul{b}\in\{1,\dots, n\}^N$, we have
\[
(u-v)[\Tsl_{kl}(u), \qmin{\ul{a}}{\ul{b}}{v}] =
\hbar\sum_{i=1}^N\left( \qmin{\ul{a}}{\rho_{i,l}(\ul{b})}{v}\cdot \Tsl_{kb_i}(u)
-\Tsl_{a_il}(u)\cdot \qmin{\rho_{i,k}(\ul{a})}{\ul{b}}{v} \right),
\]
\begin{multline*}
(u-v-\hbar(N-1))[\Tsl_{kl}(u), \qmin{\ul{a}}{\ul{b}}{v}] =
\hbar\sum_{i=1}^N\left( \Tsl_{kb_i}(u)\cdot \qmin{\ul{a}}{\rho_{i,l}(\ul{b})}{v}\right.\\
\left.-\qmin{\rho_{i,k}(\ul{a})}{\ul{b}}{v}\cdot \Tsl_{a_il}(u) \right).
\end{multline*}
\end{prop}

\begin{pf}
Consider the generalised $RTT$ relation
\begin{align*}
R(u,v,v+\hbar,&\dots, v+\hbar(N-1))\Tsl_0(u)\Tsl_1(v)\cdots \Tsl_N(v+\hbar(N-1)) =\\
=&\Tsl_N(v+\hbar(N-1))\cdots \Tsl_1(v)\Tsl_0(u)R(u,v,v+\hbar,\dots, v+\hbar(N-1)).
\end{align*}

From the definition given in Section \ref{ssec:higher-RTT}, we get
\[
R(u,v,v+\hbar,\dots, v+\hbar(N-1)) \sim A_N+\frac{\hbar}{u-v}\sum_{i=1}^NA_NP_{0i}.
\]
The first equation then follows by applying the identity above to the vector $\bs{l}\ten \bs{\ul{b}}$ and
computing the coefficient of $\bs{k}\ten \bs{\ul{a}}$. The proof of the second one
is analogous.
\end{pf}

\begin{cor}\label{cor:minor-comm}
For $\ul{a},\ul{b}\in\{1,\dots, n\}^N$ and $1\leqslant i,j\leqslant N$,
$$[\Tsl_{a_ib_j}(u), \qmin{a_1\dots a_N}{b_1\dots b_N}{v}]=0.$$
\end{cor}

\subsection{Center of $\Usl{n}$}\label{ssec:bethe}

An easy application of Corollary \ref{cor:minor-comm} is that the coefficients
of the principal minors of $\Tsl(u)$ commute with each other. We record this
observation and the well--known fact about the center of $\Usl{n}$ below.
Recall that we defined in \eqref{eq:polyP}:
\[
\Psl_k(u) = \qmin{1,\ldots,k}{1,\ldots,k}{u-\hhalf (k-1)} \qquad
\text{for each } 1\leqslant k\leqslant n.
\] 
Note that $\Psl_k(u)$ is a (monic in $u$) homogeneous polynomial of degree $k$ in $\Usl{n}[\hbar,u]$,
where the grading is understood to be $0$ for elements of $\Usl{n}$ and $1$
for $u$ and $\hbar$. Let us denote its coefficients as:
\[
\Psl_k(u) = u^k + \sum_{j=0}^{k-1} \commelt{k}{j} \hbar^{k-j} u^{j}.
\]
We observe that $\commelt{n}{n-1} = \Tr{}\lp\Tsl\rp = 0$.


\begin{prop}\label{pr:bethe}
\hfill
\begin{enumerate}
\item The elements $\{\commelt{k}{j}\}_{1\leqslant k\leqslant n, 0\leqslant j\leqslant k-1}$ form a commutative
subalgebra of $\Usl{n}$.

\item The elements $\{\commelt{n}{j}\}_{0\leqslant j\leqslant n-2}$ are algebraically independent
and generate the center of $\Usl{n}$.
\[
\mathcal{Z}(\Usl{n}) = \C[\commelt{n}{0},\ldots,\commelt{n}{n-2}].
\]
\item The roots of $\Psl_k(u)$ are distinct.
\end{enumerate}
\end{prop}

\begin{pf}
As remarked earlier, (1) follows directly from Corollary \ref{cor:minor-comm}.
We briefly sketch the proof of (2) which is otherwise well--known. 
One identifies the center $\mathcal{Z}(\Usl{n})$ with the algebra of invariants
$\C[\h]^{\fkS_n}$ (see, for instance, \cite[\S 7.4]{dixmier}). The latter
is a polynomial ring in $n-1$ variables, $p_2,\ldots,p_{n}$ (power sum
symmetric functions). The reader can easily check that under these
identifications $\commelt{n}{j} = p_{j+2}$ which proves the claimed assertion.\\

(3) follows from (2) using the following standard argument. 
Let $P(u)$ be a monic polynomial with
coefficients from a (unital) commutative ring $A$. 
Define $\operatorname{Disc}(P) = \prod_{i\neq j} (a_i-a_j)$ where $\{a_i\}$ are the roots
of $P(u)$. Note that the expression of $\operatorname{Disc}(P)$ is symmetric in $\{a_i\}$
and hence it is a polynomial in the coefficients of $P(u)$. By definition
$\operatorname{Disc}(P)=0$ if, and only if $P(u)$ has some root with multiplicity $>1$.
In this case one obtains a non--trivial algebraic relation among the coefficients
of $P(u)$.
\end{pf}

\subsection{$\psi$--operators}\label{ssec:psi-operator}

Let $0\leqslant k\leqslant n-2$ and $m=n-k$. For each $i,j\in\{1,\ldots, m\}$ define
\begin{equation}\label{eq:psi-operator}
\psiop{k}{\Tsl}_{ij}(u) := \qmin{1,\ldots,k}{1,\ldots,k}{u-\hhalf k}^{-1}
\qmin{1,\ldots,k,k+i}{1,\ldots,k,k+j}{u-\hhalf k}.
\end{equation}

We will skip the dependence on $\Tsl$ from the notation when no confusion is possible.
We view this $\psiopw{k}_{ij}(u)$ as an element of $\Usl{n}[\hbar][u;u^{-1}]]$.

\begin{prop}\label{pr:psi-operator}\hfill
\begin{enumerate}
\item The $m\times m$ matrix $\psiopw{k}$ satisfies the $\RTT$ relations
with $R(u) = u\Id + \hbar P \in \auxaux{m}[\hbar,u]$. More explicitly, the following
relations hold for $a,b,c,d\in\{1,\ldots,m\}$
\begin{equation}
(u-v)[\psiopw{k}_{ab}(u),\psiopw{k}_{cd}(v)] = \hbar
\lp
\psiopw{k}_{ad}(u)\psiopw{k}_{cb}(v) - 
\psiopw{k}_{ad}(v)\psiopw{k}_{cb}(u) 
\rp.
\end{equation}

\item The following iteration relation holds for the $\psi$--operator
\[
\psiop{k}{\psiop{l}{\Tsl}} = \psiop{k+l}{\Tsl}.
\]
\end{enumerate}
\end{prop}

\begin{pf}
We prove (2) first. For that it is enough to prove the assertion for $k=1$ (the general
case follows from repeated application of $k=1$ case). 
Let us assume that we have a matrix $\phi(u) \in \aux{m}\otimes 
\Usl{n}[\hbar][u;u^{-1}]]$ satisfying the $\RTT$ relations.
The reader can verify
easily that the equation $\psiop{1}{\psiop{l}{\phi}} = \psiop{l+1}{\phi}$,
is equivalent to the following determinant identity:
\begin{multline}\label{eq:pfpsiop1}
\tqmin{1,\ldots,l,l+1,a}{1,\ldots,l,l+1,b}{u}
\tqmin{1,\ldots,l}{1,\ldots,l}{u+\hbar} = 
\tqmin{1,\ldots,l,l+1}{1,\ldots,l,l+1}{u}
\tqmin{1,\ldots,l,a}{1,\ldots,l,b}{u+\hbar} -\\
\tqmin{1,\ldots,l,a}{1,\ldots,l,l+1}{u}
\tqmin{1,\ldots,l,l+1}{1,\ldots,l,b}{u+\hbar},
\end{multline}
where $a,b\geqslant l+2$. The proof of this identity uses Lemma \ref{lem:q-minors}.
For notational convenience we will write $\ul{l}$ for the sequence $1,\ldots,l$
and $\ul{l}\setminus i$ for the sequence $1,\ldots,\wh{i},\ldots,l$, for any
$1\leqslant i\leqslant l$. As before, let $u_j = u+j\hbar$. Then we have the following
column expansion from the second equation of Lemma \ref{lem:q-minors} (3).
\begin{multline*}
\tqmin{\ul{l};l+1,a}{\ul{l};l+1,b}{u_0} = 
\tqmin{\ul{l};l+1}{\ul{l};l+1}{u_0}\phi_{a,b}(u_{l+1})
-\tqmin{\ul{l};a}{\ul{l};l+1}{u_0}\phi_{l+1,b}(u_{l+1}) \\
+ \sum_{i=1}^l (-1)^{l+i}\tqmin{\ul{l}\setminus i; l+1,a}{\ul{l};l+1}{u_0}
\phi_{i,b}(u_{l+1}).
\end{multline*}
Substituting this expression in \eqref{eq:pfpsiop1} gives us
\begin{multline}\label{eq:pfpsiop2}
\lp\sum_{i=1}^l (-1)^{l+i}\tqmin{\ul{l}\setminus i; l+1,a}{\ul{l};l+1}{u_0}
\phi_{i,b}(u_{l+1}) \rp\cdot \tqmin{\ul{l}}{\ul{l}}{u_1} = \\
\tqmin{\ul{l+1}}{\ul{l+1}}{u_0}
\lp \tqmin{\ul{l};a}{\ul{l};b}{u_1} - \phi_{a,b}(u_{l+1})\tqmin{\ul{l}}{\ul{l}}{u_1}\rp
\\
-\tqmin{\ul{l};a}{\ul{l+1}}{u_0}
\lp \tqmin{\ul{l+1}}{\ul{l};b}{u_1} - \phi_{l+1,b}(u_{l+1})
\tqmin{\ul{l}}{\ul{l}}{u_1}\rp.
\end{multline}
Now we use (2) of Lemma \ref{lem:q-minors} and the row expansion formula
(the fourth equality of Lemma \ref{lem:q-minors} (3)):
\begin{multline*}
\tqmin{\ul{l};\alpha}{\ul{l};\beta}{w} = \tqmin{\alpha;\ul{l}}{\beta;\ul{l}}{w}
= \phi_{\alpha,\beta}(w_l)\tqmin{\ul{l}}{\ul{l}}{w} + \\
\sum_{j=1}^l (-1)^{l+j+1}\phi_{\alpha,j}(w_l)
\tqmin{\ul{l}}{\ul{l}\setminus j; \beta}{w}
\end{multline*}
to rewrite the \rhs of \eqref{eq:pfpsiop2} as
\begin{align*}
\text{R.H.S.} &= \tqmin{\ul{l+1}}{\ul{l+1}}{u_0}\lp
\sum_{j=1}^l (-1)^{l+j+1}\phi_{a,j}(u_{l+1})\tqmin{\ul{l}}{\ul{l}\setminus j; b}{u_1}
\rp \\
&\phantom{=} - \tqmin{\ul{l};a}{\ul{l+1}}{u_0}\lp
\sum_{j=1}^l (-1)^{l+j+1}\phi_{l+1,j}(u_{l+1})\tqmin{\ul{l}}{\ul{l}\setminus j; b}{u_1}
\rp \\
&= \sum_{j=1}^l (-1)^{l+j+1}
\lp
\tqmin{\ul{l+1}}{\ul{l+1}}{u_0}\phi_{a,j}(u_{l+1}) 
\right.\\ & \qquad\qquad\left. - 
\tqmin{\ul{l};a}{\ul{l+1}}{u_0}\phi_{l+1,j}(u_{l+1})
\rp
\tqmin{\ul{l}}{\ul{l}\setminus j; b}{u_1}\\
&= \sum_{i,j=1}^l (-1)^{i+j}
\tqmin{\ul{l}\setminus i; l+1,a}{\ul{l};l+1}{u_0}
\phi_{ij}(u_{l+1})
\tqmin{\ul{l}}{\ul{l}\setminus j; b}{u_1},
\end{align*}
where in the last equality we used the column expansion of a matrix with a repeated
column:
\begin{multline*}
0 = \tqmin{\ul{l};l+1,a}{\ul{l};l+1,j}{w} = 
\tqmin{\ul{l};l+1}{\ul{l};l+1}{w}\phi_{a,j}(w_{l+1})
- \tqmin{\ul{l};a}{\ul{l};l+1}{w}\phi_{l+1,j}(w_{l+1}) \\
+ \sum_{i=1}^l (-1)^{l+i}\tqmin{\ul{l}\setminus i; l+1,a}{\ul{l};l+1}{w}
\phi_{ij}(w_{l+1}).
\end{multline*}
This turns the equation \eqref{eq:pfpsiop2} that we need to verify into the following
\begin{multline}\label{eq:pfpsiop3}
\sum_{i=1}^l (-1)^i \tqmin{\ul{l}\setminus i; l+1,a}{\ul{l};l+1}{u_0}\cdot \\
\lp
\sum_{j=1}^l (-1)^j \phi_{ij}(u_{l+1})\tqmin{\ul{l}}{\ul{l}\setminus j; b}{u_1}
-(-1)^l\phi_{i,b}(u_{l+1})
\tqmin{\ul{l}}{\ul{l}}{u_1}
\rp = 0.
\end{multline}
Now we only need to observe that each $i^{\scriptstyle{\text{th}}}$ term in the
equation above is the row expansion of $\tqmin{i;\ul{l}}{\ul{l};b}{u_1}$
which is zero since $i\in \{1,\ldots,l\}$. This finishes the proof of (2).\\

In order to prove (1) we observe that an easy induction argument, using (2),
reduces it to the case of $k=1$. Again we revert back to a more general set up
where we are given a matrix $\phi(u) \in \aux{m}\otimes \Usl{n}[\hbar][u;u^{-1}]]$
satisfying the $\RTT$ relation. We need to prove the following equation
(see \eqref{eq:Xu-entry-rel}):
\begin{equation}\label{eq:pfpsiop3-1}
(u-v)[\psi_{ac}(u),\psi_{bd}(v)] = \hbar (
\psi_{bc}(v)\psi_{ad}(u) - \psi_{bc}(u)\psi_{ad}(v)),
\end{equation}
where $\psi_{ij}(u) = \tqmin{1i}{1j}{u}$ for any $i,j\in \{2,\ldots,m\}$. Our
proof is based on the idea behind Proposition \ref{pr:minor-comm}. Namely,
we take the $R$--matrix $R(u,u+\hbar,v,v+\hbar)$ in $\haux{m}{4}$ and use 
the generalisation of the $\RTT$ relations given in \eqref{eq:gen-RXX} for $N=4$:
\begin{multline*}
R(u,u+\hbar,v,v+\hbar)\phi_1(u)\phi_2(u+\hbar)\phi_3(v)\phi_4(v+\hbar) = \\
\phi_4(v+\hbar)\phi_3(v)\phi_2(u+\hbar)\phi_1(u)R(u,u+\hbar,v,v+\hbar).
\end{multline*}
We will apply this operator to the basis vector $\bs{1c1d}$ and compute
the coefficient of $\bs{1a1b}$. For this, we rewrite $R(u,u+\hbar,v,v+\hbar)$
using the Yang--Baxter equation \eqref{eq:ybe}, where $x=u-v$:
\begin{align*}
R(u,u+\hbar,v,v+\hbar) &= \lp R_{34}(-\hbar)R_{12}(-\hbar)\rp
R_{14}(x-\hbar)R_{24}(x)R_{13}(x)R_{23}(x+\hbar) \\
&= R_{23}(x+\hbar)R_{24}(x)R_{13}(x)R_{14}(x-\hbar)
\lp R_{34}(-\hbar)R_{12}(-\hbar)\rp
\end{align*}

Thus the operator we are interested in, say $\mathcal{T}(u,v)$, takes the following form.
\begin{align*}
\mathcal{T}(u,v) &= \lb \phi_4(v+\hbar)\phi_3(v)R_{34}(-\hbar)\rb \cdot 
\lb \phi_2(u+\hbar)\phi_1(u)R_{12}(-\hbar)\rb \cdot \\
&\phantom{=}\qquad\qquad 
R_{14}(x-\hbar)R_{24}(x)R_{13}(x)R_{23}(x+\hbar)\\
&= R_{23}(x+\hbar)R_{24}(x)R_{13}(x)R_{14}(x-\hbar)\cdot \\
&\phantom{=}\qquad\qquad
\lb R_{12}(-\hbar)\phi_1(u)\phi_2(u+\hbar)\rb\cdot
\lb R_{34}(-\hbar)\phi_3(v)\phi_4(v+\hbar)\rb
\end{align*}

Thus the coefficient of $\bs{1a1b}$ in $\mathcal{T}(u,v)\bs{1c1d}$ computed using the first
expression of $\mathcal{T}(u,v)$ gives the following answer, using the definition
of quantum minors given in \eqref{eq:q-minors}:
\begin{multline}\label{eq:pfpsiop4}
\left\langle 1a1b\right| \mathcal{T}(u,v) \bs{1c1d} = (\hbar x)^2 (x^2-\hbar^2)
\lp\frac{x+2\hbar}{x+\hbar}\rp\lp
\tqmin{1b}{1d}{v}\tqmin{1a}{1c}{u} \right.\\
\left.
+ \frac{\hbar}{x} \tqmin{1b}{1c}{v}\tqmin{1a}{1d}{u}
\rp.
\end{multline}

Similarly, using the second expression of $\mathcal{T}(u,v)$, the same coefficient
turns out to be:
\begin{multline}\label{eq:pfpsiop5}
\left\langle 1a1b\right| \mathcal{T}(u,v) \bs{1c1d} = (\hbar x)^2 (x^2-\hbar^2)
\lp\frac{x+2\hbar}{x+\hbar}\rp\lp
\tqmin{1a}{1c}{u}\tqmin{1b}{1d}{v} \right.\\
\left.
+ \frac{\hbar}{x} \tqmin{1b}{1c}{u}\tqmin{1a}{1d}{v}
\rp.
\end{multline}

Combining equations \eqref{eq:pfpsiop4} and \eqref{eq:pfpsiop5} we get the desired equation
\eqref{eq:pfpsiop3-1}.
\end{pf}


\section{The evaluation homomorphism}\label{sec:eval}

In this section, we complete the proof of Theorem \ref{thm:main1} by describing the evaluation
homomorphism from the Yangian $\Yhsl{n}$ to ${U\sl_n}[\hbar]$ with respect to the loop generators
of the Yangian.

\subsection{Evaluation homomorphism}\label{ssec:eval}

Recall the definition of the generating series $\{\xi_k(u),x_k^{\pm}(u)\}_{k\in\bfI}$
from Section \ref{ssec:yfields}. Let $\Psl_k(u)$ be given by \eqref{eq:polyP}, \ie
\[
\Psl_k(u) = \qmin{1,\ldots,k}{1,\ldots,k}{u-\hhalf(k-1)}.
\]
and define
\begin{align}
\ev(\xi_k(u)) &= \frac{\Psl_{k-1}(u)\Psl_{k+1}(u)}{\Psl_k\lp u+\hhalf\rp \Psl_k\lp u-\hhalf\rp}
\label{eq:evxi}\\
\ev(x_k^+(u)) &= \Psl_k\lp u+\hhalf\rp^{-1}\left[e_{k,k+1},\Psl_k\lp u+\hhalf\rp\right] 
\label{eq:evx+}\\
\ev(x_k^-(u)) &= \Psl_k\lp u-\hhalf\rp^{-1}\left[\Psl_k\lp u-\hhalf\rp, e_{k+1,k}\right] 
\label{eq:evx-}
\end{align}

\begin{thm}\label{thm:evishom}
The formulae above define an algebra homomorphism $$\ev \colon \Yhsl{n} \to \Usl{n}[\hbar].$$
\end{thm}

\subsection*{Remark}
We are grateful to Maxim Nazarov for pointing out that the defining formulae of the morphism
$\ev$ are similar to those appearing in 
\cite{brundan-kleshev, drinfeld-yangian-qaffine, molev-yangian, nazarov-tarasov-rep-gt}, 
which define an embedding $\iota\colon\Yhsl{n}\to\Yhgl{n}$ and describe the Drinfeld
generators of $\Yhsl{n}$ in terms of quantum minors in $\Yhgl{n}$.
The relation with the homomorphism $\ev$ is easily explained. Since the matrix 
$\Tsl(u)$ satisfies the $RTT$ relation \eqref{eq:RXX}, there is an induced algebra 
homomorphism $\ev_{\Tsl}\colon\Yhgl{n}\to\Usl{n}[\hbar]$. Then, $\ev=\ev_{\Tsl}\circ\iota$.\\

 
\begin{pf}
We note that the coefficients of the polynomials $\{\Psl_k(u)\}_{1\leqslant k\leqslant n}$ commute,
because of Corollary \ref{cor:minor-comm}. Thus we get $[\xi_i(u),\xi_j(v)]=0$ for
every $i,j\in\bfI$ which is the relation ($\mathcal{Y}$1) of Section \ref{ssec:yfields}.\\

Comparing the coefficients of $\hbar u^{-1}$ in the definition of $\ev$,
and observing that
\[
\Psl_k(u) = u^k -\hbar\cowt_{k} u^{k-1} + \cdots
\]
it follows that $\ev(\xi_{k,0}) = 2\cowt_{k}-\cowt_{k-1}-\cowt_{k+1} = h_k$,
$\ev(x_{k,0}^+) = e_{k,k+1}$ and $\ev(x_{k,0}^-) = e_{k+1,k}$. Thus the relation (Y6)
of Section \ref{ssec:ysln} holds for $r_1=\cdots=r_m=s=0$ from the Serre
relations defining $\Usl{n}$. In turn this special case of (Y6) implies the
general case (see remark preceding Lemma \ref{lem:Y2prime}).\\

Our proof of the rest of the relations uses the $\psi$--operator introduced
in Section \ref{ssec:psi-operator} and the expression of $\ev$ in using $\psi$,
as proved below in Section \ref{ssec:evisrecursive}.\\

The proofs of ($\mathcal{Y}$2), ($\mathcal{Y}$3) and
($\mathcal{Y}$4) are given below in Sections \ref{ssec:pfev-y2}, \ref{ssec:pfev-y3}
and \ref{ssec:pfev-y4} respectively.
\end{pf}

\subsection{$\psi$--operators and evaluation homomorphism}\label{ssec:evisrecursive}

We begin by making the observation that the definition of the evaluation
map $\ev$ is obtained {\em recursively} using the $\psi$--operator introduced
in Section \ref{ssec:psi-operator}. To state this precisely, we begin by 
rewriting \eqref{eq:evx+} and \eqref{eq:evx-} using the following easily
verified identities:
\begin{eqnarray*}
\left[e_{k,k+1}, \Psl_{k}\lp u+\frac{\hbar}{2}\rp\right] &=&
-\qmin{1,\ldots, k}{1,\ldots, \wh{k},k+1}{ u-\frac{\hbar}{2}(k-1)
+\frac{\hbar}{2}},\\
\left[\Psl_{k}\lp u-\frac{\hbar}{2}\rp, e_{k+1,k}\right] &=& 
-\qmin{1,\ldots, \wh{k},k+1}{1,\ldots, k}{u-\frac{\hbar}{2}(k-1)
-\frac{\hbar}{2}}.
\end{eqnarray*}

Note that the formulae \eqref{eq:evxi}, \eqref{eq:evx+} and \eqref{eq:evx-} for $k=1$
take the following form:
\begin{align}
\ev(\xi_1(u)) &= \Tsl_{11}\lp u+\hhalf\rp^{-1}\Tsl_{11}\lp u-\hhalf\rp^{-1}
\qmin{12}{12}{u-\hhalf} \label{eq:evxi1},\\
\ev(x_1^+(u)) &= -\Tsl_{11}\lp u+\hhalf\rp^{-1} \Tsl_{12}\lp u+\hhalf\rp 
\label{eq:evx+1},\\
\ev(x_1^-(u)) &= -\Tsl_{11}\lp u-\hhalf\rp^{-1} \Tsl_{21}\lp u-\hhalf\rp 
\label{eq:evx-1}.
\end{align}

\begin{lem}\label{lem:recursive-rel}
For each $k\geqslant 1$, we have
\begin{align}
\ev(\xi_k(u)) &= \psiopw{k-1}_{11}\lp u+\hhalf\rp^{-1}\psiopw{k-1}_{11}\lp u-\hhalf\rp^{-1}
\qminpsi{k-1}{12}{12}{u-\hhalf} \label{eq:evxik},\\
\ev(x_k^+(u)) &= -\psiopw{k-1}_{11}\lp u+\hhalf\rp^{-1} \psiopw{k-1}_{12}\lp u+\hhalf\rp 
\label{eq:evx+k},\\
\ev(x_k^-(u)) &= -\psiopw{k-1}_{11}\lp u-\hhalf\rp^{-1} \psiopw{k-1}_{21}\lp u-\hhalf\rp 
\label{eq:evx-k}.
\end{align}
\end{lem}

\begin{pf}
The proof of this lemma is a direct verification, which we carry out below.
Let us start with the assertion for $x_k^+(u)$ from \eqref{eq:evx+k}. We 
expand the \rhs of this equation:
\begin{align*}
\text{R.H.S.} 
& = -\qmin{1,\ldots, k}{1,\ldots, k}{u-\hhalf(k-1)+\hhalf}^{-1}\cdot
\qmin{1,\ldots, k-1}{1,\ldots, k-1}{u-\hhalf(k-1)+\hhalf}\cdot \\
&\cdot\qmin{1,\ldots, k-1}{1,\ldots, k-1}{u-\hhalf(k-1)+\hhalf}^{-1}
\cdot\qmin{1,\ldots, k}{1,\ldots,\wh{k},k+1}{u-\hhalf(k-1)+\hhalf}\\
& = \Psl_{k}\lp u+\hhalf\rp^{-1}\cdot
\left[e_{k,k+1}, \Psl_{k}\lp u+\frac{\hbar}{2}\rp\right]=\ev(x^+_k(u)).
\end{align*}

Now consider the \rhs of \eqref{eq:evx-k}:
\begin{align*}
\text{R.H.S.} 
& = -\qmin{1,\ldots, k}{1,\ldots, k}{u-\hhalf(k-1)-\hhalf}^{-1}
\cdot\qmin{1,\ldots, k-1}{1,\ldots, k-1}{u-\hhalf(k-1)-\hhalf}\cdot \\
&\cdot\qmin{1,\ldots, k-1}{1,\ldots, k-1}{u-\hhalf(k-1)-\hhalf}^{-1}
\cdot\qmin{1,\ldots, k}{1,\ldots,\wh{k},k+1}{u-\hhalf(k-1)-\hhalf}\\
& = \Psl_{k}\lp u-\hhalf\rp^{-1}\cdot\left[\Psl_{k}\left(u-\hhalf\right), e_{k+1,k}\right]=
\ev(x_k^-(u)).
\end{align*}

Finally for $\xi_k(u)$, the \rhs of \eqref{eq:evxik} can be expanded as below:
\begin{align*}
\text{R.H.S.} 
&=\qmin{1,\ldots, k-1}{1,\ldots, k-1}{u-\hhalf(k-1)+\hhalf}
\cdot\qmin{1,\ldots, k}{1,\ldots, k}{u-\hhalf(k-1)+\hhalf}^{-1}\cdot\\
&\cdot\qmin{1,\ldots, k-1}{1,\ldots, k-1}{u-\hhalf(k-1)-\hhalf}
\cdot\qmin{1,\ldots, k}{1,\ldots, k}{u-\hhalf(k-1)-\hhalf}^{-1}\cdot\\
&\cdot\qmin{1,\ldots, k-1}{1,\ldots, k-1}{u-\hhalf(k-1)-\hhalf}^{-1}
\cdot\qmin{1,\ldots, k+1}{1,\ldots, k+1}{u-\hhalf(k-1)-\hhalf}\\
& = 
\frac
{\qmin{1,\ldots, k-1}{1,\ldots, k-1}{u-\hhalf(k-2)}\cdot
\qmin{1,\ldots, k+1}{1,\ldots, k+1}{u-\hhalf k}}
{\qmin{1,\ldots, k}{1,\ldots, k}{u-\hhalf(k-1)+\hhalf}\cdot
\qmin{1,\ldots, k}{1,\ldots, k}{u-\hhalf(k-1)-\hhalf}}\\
& =
\frac
{\Psl_{k-1}(u)\Psl_{k+1}(u)}
{\Psl_{k}\left(u+\hhalf\right)\Psl_{k}\left(u-\hhalf\right)}
=\ev\lp\xi_k(u)\rp.
\end{align*}
\end{pf}

\subsection{Proof of ($\mathcal{Y}$2)}\label{ssec:pfev-y2}

Using Proposition \ref{pr:reduction}, it suffices to prove the following
two relations:
\begin{itemize}
\item For each $i\in\bfI$
\begin{equation}\label{eq:pfev-y2-1}
\Ad(\xi_i(u))^{-1}(x_i^{\pm}(v)) = \frac{u-v\mp \hbar}{u-v\pm \hbar}
x_i^{\pm}(v) \pm \frac{2\hbar}{u-v\pm \hbar} x_j^{\pm}(u\pm \hbar).
\end{equation}

\item For each $i\neq j\in\bfI$
\begin{equation}\label{eq:pfev-y2-2}
\lb \xi_i(u), x_{j,0}^{\pm}\rb = \pm a_{ij} \xi_i(u)x_j^{\pm}\lp u\pm\hhalf a_{ij}\rp.
\end{equation}

\end{itemize}

Below we prove these for the $+$ case, for definiteness. The $-$ case is entirely
analogous.\\

The equation \eqref{eq:pfev-y2-2} for $j\not\in \{i-1, i+1\}$ holds,
since in that case $e_{j,j+1}$ commutes with $\{\Psl_{i-1}, \Psl_i, \Psl_{i+1}\}$
defining $\ev(\xi_i(u))$. For $j=i+1$ one only has to observe that
$e_{i+1,i+2}$ commutes with $\Psl_{i-1}$ and $\Psl_i$. Similarly the case
$j=i-1$.\\

Thus we are left with proving \eqref{eq:pfev-y2-1} for any $i\in\bfI$. Using
Lemma \ref{lem:recursive-rel}, we can write $\ev(\xi_i(u))$ and $\ev(x_i^+(u))$
in terms of the $\psi$--operator. Below we omit the superscript $(i-1)$ from 
$\psiopw{i-1}(u)$.
\begin{align*}
\ev(\xi_i(u)) &= \psiopww_{11}\lp u+\hhalf\rp^{-1} \psiopww_{11}\lp u-\hhalf\rp^{-1}
\qminpsiw{12}{12}{u-\hhalf},\\
\ev(x_i^+(u)) &= -\psiopww_{11}\lp u+\hhalf\rp^{-1} \psiopww_{12}\lp u+\hhalf\rp.
\end{align*}

\noindent Note that $\lb\psiopww(u),\psiopww(v)\rb = 0$ and  we have the following relation 
between $\psiopww_{11}$ and $\psiopww_{12}$
\begin{equation}\label{eq:pfev-y2-3}
\Ad\lp\psiopww_{11}(u)\rp\cdot \psiopww_{12}(v) = 
\frac{u-v-\hbar}{u-v} \psiopww_{12}(v) + \frac{\hbar}{u-v}\psiopww_{12}(u)
\psiopww_{11}(v)\psiopww_{11}(u)^{-1}.
\end{equation}

Setting $u=v+\hbar$ in this equation gives the following:
\begin{equation}\label{eq:pfev-y2-4}
\psiopww_{12}(v)\psiopww_{11}(v)^{-1} = 
\psiopww_{11}(v+\hbar)^{-1}\psiopww_{12}(v+\hbar).
\end{equation}

Combining these observation, the equation we need to verify, namely
\eqref{eq:pfev-y2-1} takes the following form, after using
\eqref{eq:pfev-y2-4} and renaming variables
$u,v \mapsto u+\hhalf, v+\hhalf$:
\begin{multline*}
\Ad\lp\psiopww_{11}(u-\hbar)\psiopww_{11}(u)\rp\cdot \psiopww_{12}(v-\hbar) = 
\frac{u-v-\hbar}{u-v+\hbar}\psiopww_{12}(v-\hbar) + \\
\frac{2\hbar}{u-v+\hbar}\psiopww_{12}(u)\psiopww_{11}(u)^{-1}
\psiopww_{11}(v-\hbar),
\end{multline*}
which is a direct consequence of repeated application of
\eqref{eq:pfev-y2-3}.

\subsection{Proof of ($\mathcal{Y}$3)}\label{ssec:pfev-y3}

Recall that we need to prove the following relation for every pair $i,j\in\bfI$.
\begin{multline}\label{eq:pfev-y3-0}
(u-v\mp a)x_i^{\pm}(u)x_j^{\pm}(v) = (u-v\pm a)x_j^{\pm}(v)x_i^{\pm}(u)
\\ +\hbar \lp \lb x_{i,0}^{\pm}, x_j^{\pm}(v)\rb - \lb x_i^{\pm}(u),
x_{j,0}^{\pm}\rb \rp.
\end{multline}
For $i=j$ or for a pair with $a_{ij}=0$, it suffices to prove the
following special case (see Proposition \ref{pr:reduction}).
\begin{equation}\label{eq:pfev-y3-1}
[x_{i,0}^{\pm},x_j^{\pm}(u)] - [x_i^{\pm}(u), x_{j,0}^{\pm}]
= \mp \frac{a_{ij}}{2} (x_i^{\pm}(u)x_j^{\pm}(u) + x_j^{\pm}(u)x_i^{\pm}(u)).
\end{equation}

Let us prove this relation for the $+$ case only. Note that for
$i,j\in\bfI$ such that $a_{ij}=0$, this relation follows from $[x_i^+(u),x_{j,0}^+]=0$
which is true since $e_{j,j+1}$ commutes with $e_{i,i+1}$ and $\Psl_i$.\\

Next, let us assume $i=j$. In this case we need to show that
$[x_{i,0}^+,x_i^+(u)] = -x_i^+(u)^2$. Below, we will use the fact that $e_{i,i+1}$
commutes with the commutator $[e_{i,i+1},\Psl_i(u)]$. This is because
this commutator can be written as a quantum--minor:
\[
[e_{i,i+1},\Psl_i(u)] = -\qmin{1,\ldots,i}{1,\ldots,i-1,i+1}{u-\hhalf(i-1)},
\]
and $e_{i,i+1}$ is an entry of the indicated submatrix, and we can use
Corollary \ref{cor:minor-comm} to conclude that it commutes with the quantum--minor.
Thus we have the following computation, with $\wt{u} = u-\hhalf$ for convenience:
\begin{align*}
[x_{i,0}^+,x_i^+(\wt{u})] &= \lb e_{i,i+1}, \Psl_i(u)^{-1}[e_{i,i+1},\Psl_i(u)]\rb \\
&= -\Psl_i(u)^{-1}[e_{i,i+1},\Psl_i(u)]\Psl_i(u)^{-1}[e_{i,i+1},\Psl_i(u)] \\
&= -x_i^+(\wt{u})^2
\end{align*}
as intended. Note that we used the identity $[\alpha,\beta^{-1}] = 
-\beta^{-1}[\alpha,\beta]\beta^{-1}$ in the calculation above.\\

Finally, we are left with the case $j=i+1$. We will reduce this case to rank $2$
using the $\psi$--operator. To do this, we need to rewrite the commutators
on the \rhs of \eqref{eq:pfev-y3-0} as follows. For convenience, below we
write $\wt{u} = u-\hhalf$ and $\wt{v} = v-\hhalf$. Using the definition,
and the fact that $e_{i,i+1}$ commutes with $\Psl_{i+1}$ (see Corollary
\ref{cor:minor-comm}) we get
\begin{align*}
[x_{i,0}^+, x_{i+1}^+(\wt{v})] &= \Psl_{i+1}(v)^{-1}
\lb e_{i,i+2},\Psl_{i+1}(v)\rb \\
&= \qmin{1,\ldots,i+1}{1,\ldots,i+1}{v-\hhalf i}^{-1} \cdot
\qmin{1,\ldots,i+1}{1,\ldots,i-1,i+1,i+2}{v-\hhalf i}.
\end{align*}
Similarly we get
\[
[x_i^+(\wt{u}),x_{i+1,0}^+] = - \qmin{1,\ldots,i}{1,\ldots,i}{u-\hhalf(i-1)}^{-1}\cdot
\qmin{1,\ldots,i}{1,\ldots,i-1,i+2}{u-\hhalf(i-1)}.
\]

Clearing the inverses of the principal quantum--minors from both sides of
\eqref{eq:pfev-y3-0} and using the properties of the $\psi$--operator from
Proposition \ref{pr:psi-operator}, we get the following version of \eqref{eq:pfev-y3-0}:
\begin{multline*}
\lp u-v+\hhalf\rp \psiopww_{12}(u)\qminpsiw{12}{13}{v-\hhalf} - 
\lp u-v-\hhalf\rp \qminpsiw{12}{13}{v-\hhalf}\psiopww_{12}(u) = \\
\hbar\lp
\psiopww_{11}(u)\qminpsiw{12}{23}{v-\hhalf} + 
\qminpsiw{12}{12}{v-\hhalf}\psiopww_{13}(v)
\rp.
\end{multline*}

Rearranging the terms of this equation, and replacing $v-\hhalf$ by $v$, we get
the following equation that we need to verify:
\begin{multline}\label{eq:pfev-y3-2}
(u-v-\hbar)\lb \psiopww_{12}(u), \qminpsiw{12}{13}{v}\rb = 
\\
\hbar\lp
\psiopww_{11}(u)\qminpsiw{12}{23}{v} + \qminpsiw{12}{12}{v}\psiopww_{13}(u)
-\psiopww_{12}(u)\qminpsiw{12}{13}{v}
\rp.
\end{multline}

Since the $\psi$--matrix also satisfies the $\RTT$ relations (see Proposition
\ref{pr:psi-operator}) we can use the commutation relations derived in 
Proposition \ref{pr:minor-comm}. Using the second identity given there, with
$N=2$ and $k=1,l=2,a_1=1,a_2=2,b_1=1,b_2=3$, we get
\begin{multline*}
(u-v-\hbar)\lb \psiopww_{12}(u), \qminpsiw{12}{13}{v}\rb = 
\\
\hbar\lp
\psiopww_{11}(u)\qminpsiw{12}{23}{v} + \psiopww_{13}(u)\qminpsiw{12}{12}{v}
-\qminpsiw{12}{13}{v}\psiopww_{12}(u)
\rp.
\end{multline*}
Thus the required relation follows from the following claim:\\

\noindent {\bf Claim.} The following equation holds:
\begin{equation}\label{eq:pfev-y3-3}
\lb \psiopww_{13}(u),\qminpsiw{12}{12}{v}\rb + 
\lb \psiopww_{12}(u),\qminpsiw{12}{13}{v}\rb  = 0.
\end{equation}

\noindent {\em Proof of the claim.} Multiply the \lhs by $(u-v)$
and use the first relation given in Proposition \ref{pr:minor-comm} to get
\begin{align*}
\lb \psiopww_{13}(u),\qminpsiw{12}{12}{v}\rb &= \hbar\lp
-\qminpsiw{12}{23}{v}\psiopww_{11}(u) + \qminpsiw{12}{13}{v}\psiopww_{12}(u)
\right. \\ & \qquad\qquad \left.
-\psiopww_{13}(u)\qminpsiw{12}{12}{v}\rp,\\
\lb \psiopww_{12}(u),\qminpsiw{12}{13}{v}\rb &= \hbar\lp
\qminpsiw{12}{23}{v}\psiopww_{11}(u) + \psiopww_{12}(u)\qminpsiw{12}{13}{v}
\right. \\ & \qquad\qquad \left.
-\qminpsiw{12}{12}{v}\psiopww_{13}(u)\rp.
\end{align*}

Adding the two, we get that
\begin{multline*}
(u-v)\lp\lb \psiopww_{13}(u),\qminpsiw{12}{12}{v}\rb + 
\lb \psiopww_{12}(u),\qminpsiw{12}{13}{v}\rb\rp = 
\\
-\hbar\lp \lb \psiopww_{13}(u),\qminpsiw{12}{12}{v}\rb + 
\lb \psiopww_{12}(u),\qminpsiw{12}{13}{v}\rb\rp.
\end{multline*}
This prove the claim and the relation ($\mathcal{Y}$3).

\subsection{Proof of ($\mathcal{Y}$4)}\label{ssec:pfev-y4}

Again using Proposition \ref{pr:reduction}, it is sufficient to prove the
following two versions of ($\mathcal{Y}$4).
\begin{itemize}
\item For each $i\in\bfI$, we have
\begin{equation}\label{eq:pfev-y4-1}
(u-v)[x_i^+(u),x_i^-(v)] = \hbar (\xi_i(v)-\xi_i(u)).
\end{equation}
\item For $i\neq j$, we have 
\begin{equation}\label{eq:pfev-y4-2}
[x_i^+(u),x_{j,0}^-] = 0.
\end{equation}
\end{itemize}

Note that \eqref{eq:pfev-y4-2} follows easily since $e_{j+1,j}$ commutes
with $\Psl_i$ for $i\neq j$. We will now prove \eqref{eq:pfev-y4-1} using
the $\psi$--operator as before. Recall that by definition, we have
(again we omit the superscript $(i-1)$ from $\psiopw{i-1}(u)$).
\begin{align*}
\ev(x_i^+(u)) &= -\psiopww_{11}\lp u+\hhalf\rp^{-1}\psiopww_{12}\lp u+\hhalf\rp, \\
\ev(x_i^-(u)) &= -\psiopww_{11}\lp u-\hhalf\rp^{-1}\psiopww_{21}\lp u-\hhalf\rp.
\end{align*}

In order to carry out the proof, we will need to use the following
relations:
\begin{align}
\Ad\lp\psiopww_{11}(u)\rp^{-1}\cdot \psiopww_{12}(v) &= 
\frac{u-v+\hbar}{u-v} \psiopww_{12}(v)
-\frac{\hbar}{u-v} \psiopww_{11}(u)^{-1}\psiopww_{11}(v)\psiopww_{12}(u),
\label{eq:pfev-y4-3}\\
\Ad\lp\psiopww_{11}(u)\rp^{-1}\cdot \psiopww_{21}(v) &= 
\frac{u-v-\hbar}{u-v} \psiopww_{21}(v)
+\frac{\hbar}{u-v} \psiopww_{11}(u)^{-1}\psiopww_{11}(v)\psiopww_{21}(u),
\label{eq:pfev-y4-4}
\end{align}
\begin{align}
\psiopww_{12}(v)\psiopww_{11}(v)^{-1} &= 
\psiopww_{11}(v+\hbar)^{-1}\psiopww_{12}(v+\hbar)
\label{eq:pfev-y4-5},\\
\psiopww_{21}(v)\psiopww_{11}(v)^{-1} &= 
\psiopww_{11}(v-\hbar)^{-1}\psiopww_{21}(v-\hbar)
\label{eq:pfev-y4-6},\\
(u-v)\lb \psiopww_{12}(u), \psiopww_{21}(v)\rb &= 
\hbar (\psiopww_{11}(u)\psiopww_{22}(v) - \psiopww_{11}(v)\psiopww_{22}(v)).
\label{eq:pfev-y4-7}
\end{align}
%
\Omit{
\vspace{-0.1in}
\begin{equation}\label{eq:pfev-y4-7}
(u-v)\lb \psiopww_{12}(u), \psiopww_{21}(v)\rb = 
\hbar (\psiopww_{11}(u)\psiopww_{22}(v) - \psiopww_{11}(v)\psiopww_{22}(v)).
\end{equation}}
%
By definition, we have
\begin{multline*}
[x_i^+(u),x_i^-(v)] = 
\psiopww_{11}\lp u+\hhalf\rp^{-1}\psiopww_{12}\lp u+\hhalf\rp \psiopww_{11}\lp v-\hhalf\rp^{-1}
\psiopww_{21}\lp v-\hhalf\rp \\ 
-\psiopww_{11}\lp v-\hhalf\rp^{-1}\psiopww_{21}\lp v-\hhalf\rp \psiopww_{11}\lp u+\hhalf\rp^{-1}
\psiopww_{12}\lp u+\hhalf\rp.
\end{multline*}
To make the computation less cumbersome, let us write the equation above
as $(u-v)[x_i^+(u),x_i^-(v)] = \mathcal{T}_1(u,v) - \mathcal{T}_2(u,v)$. 
The two terms on the \rhs can be simplified 
using \eqref{eq:pfev-y4-3} and \eqref{eq:pfev-y4-4}.
\begin{align*}
\mathcal{T}_1 &= 
\psiopww_{11}\lp u+\hhalf\rp^{-1}\psiopww_{11}\lp v-\hhalf\rp^{-1}\cdot \\
& \phantom{=}
\cdot \lp
(u-v+\hbar)\psiopww_{12}\lp u+\hhalf\rp - \hbar \psiopww_{11}\lp u+\hhalf\rp
\psiopww_{12}\lp v-\hhalf\rp \psiopww_{11}\lp v-\hhalf\rp^{-1}
\rp\cdot \\
&\phantom{=} \cdot \psiopww_{21}\lp v-\hhalf\rp ,\\
\mathcal{T}_2 &= 
\psiopww_{11}\lp u+\hhalf\rp^{-1}\psiopww_{11}\lp v-\hhalf\rp^{-1}\cdot \\
& \phantom{=}
\cdot \lp
(u-v+\hbar)\psiopww_{21}\lp v-\hhalf\rp - \hbar \psiopww_{11}\lp v-\hhalf\rp
\psiopww_{21}\lp u+\hhalf\rp \psiopww_{11}\lp u+\hhalf\rp^{-1}
\rp\cdot \\
&\phantom{=} \cdot \psiopww_{12}\lp u+\hhalf\rp .
\end{align*}

Thus we get using \eqref{eq:pfev-y4-5}, \eqref{eq:pfev-y4-6} and
\eqref{eq:pfev-y4-7}, that \eqref{eq:pfev-y4-1} holds, upon carrying out
the simplification of its \lhs as follows:
\begin{align*}
\text{L.H.S.}&= 
\psiopww_{11}\lp u+\hhalf\rp^{-1}\psiopww_{11}\lp v-\hhalf\rp^{-1}\cdot
(u-v+\hbar)\lb \psiopww_{12}\lp u+\hhalf\rp, \psiopww_{21}\lp v-\hhalf\rp\rb \\
& \phantom{=} -\hbar \psiopww_{11}\lp v+\hhalf\rp^{-1}\psiopww_{11}\lp v-\hhalf\rp^{-1}
\psiopww_{12}\lp v+\hhalf\rp \psiopww_{21}\lp v-\hhalf \rp\\ 
&\phantom{=} +\hbar \psiopww_{11}\lp u+\hhalf\rp^{-1}\psiopww_{11}\lp u-\hhalf\rp^{-1}
\psiopww_{21}\lp u-\hhalf\rp \psiopww_{12}\lp u+\hhalf \rp \\
\phantom{\text{L.H.S.}}
&= \hbar \psiopww_{11}\lp v+\hhalf\rp^{-1}\psiopww_{11}\lp v-\hhalf\rp^{-1}\cdot
\lp
\psiopww_{11}\lp v+\hhalf\rp \psiopww_{22}\lp v-\hhalf\rp - \right. 
\end{align*}
\begin{align*}
\phantom{\text{L.H.S.}}
& \qquad\qquad
\left.\psiopww_{12}\lp v+\hhalf\rp \psiopww_{21}\lp v-\hhalf\rp
\rp \\
&\phantom{=} -\hbar \psiopww_{11}\lp u+\hhalf\rp^{-1}\psiopww_{11}\lp u-\hhalf\rp^{-1}\cdot
\lp
\psiopww_{11}\lp u-\hhalf\rp \psiopww_{22}\lp u+\hhalf\rp - \right. \\
& \qquad\qquad
\left.\psiopww_{21}\lp u-\hhalf\rp \psiopww_{12}\lp u+\hhalf\rp
\rp\\
&= \hbar (\xi_i(v) - \xi_i(u)).
\end{align*}

\subsection{Partial fractions}\label{ssec:partialfractions}

The following proposition is needed to compute the composition $\ev\circ\Phi$,
where $\Phi: \Uhsl{n}\to \wh{\Yhsl{n}}$ is the algebra homomorphism from 
Section \ref{ssec:homPhi}. For this, let us recall that $\{\roots{k}{1},
\ldots, \roots{k}{k}\}$ are the roots of the polynomial $\Psl_k(u)$.

\begin{prop}\label{pr:partialfractions}
For each $k\in\bfI$, we have
\begin{align*}
\ev(x_k^+(u)) &= \sum_{i=1}^k \frac{\hbar}{u+\hhalf-\roots{k}{i}} \cdot \lp
\sum_{j=1}^k (-1)^{k+j} \frac{\qmin{1,\ldots,\wh{j},\ldots,k}{1,\ldots,k-1}{\roots{k}{i}-\hhalf (k-1)}}
{\prod_{c\neq i} (\roots{k}{i}-\roots{k}{c})}
e_{j,k+1}
\rp,\\
\ev(x_k^-(u)) &= \sum_{i=1}^k \frac{\hbar}{u-\hhalf-\roots{k}{i}} \cdot \lp
\sum_{j=1}^k (-1)^{k+j} \frac{\qmin{1,\ldots,k-1}{1,\ldots,\wh{j},\ldots,k}{\roots{k}{i}-\hhalf (k-3)}}
{\prod_{c\neq i} (\roots{k}{i}-\roots{k}{c})}
e_{k+1,j}
\rp.
\end{align*}
\end{prop}

\begin{pf}
From the definition \eqref{eq:evx+}, \eqref{eq:evx-}, and the observation made
in Section \ref{ssec:evisrecursive}, we have the following:
\begin{align*}
\ev(x_k^+(u)) &= -\Psl_k\lp u+\hhalf\rp^{-1} \qmin{1,\ldots,k}{1,\ldots,k-1,k+1}{u+\hhalf
-\hhalf(k-1)},\\
\ev(x_k^-(u)) &= -\Psl_k\lp u-\hhalf\rp^{-1} \qmin{1,\ldots,k-1,k+1}{1,\ldots,k}{u-\hhalf
-\hhalf(k-1)}.
\end{align*}

Using the column and row expansions of quantum minors (first and third equations of 
Lemma \ref{lem:q-minors} (3)), we get
\begin{align*}
\qmin{1,\ldots,k}{1,\ldots,k-1,k+1}{w}&= \sum_{j=1}^k (-1)^{k+j} 
\qmin{1,\ldots,\wh{j},\ldots,k}{1,\ldots,k-1}{w} \Tsl_{j,k+1}(w+\hbar(k-1)),\\
\qmin{1,\ldots,k-1,k+1}{1,\ldots,k}{w}&= \sum_{j=1}^k (-1)^{k+j}
\qmin{1,\ldots,k-1}{1,\ldots,\wh{j},\ldots,k}{w+\hbar} \Tsl_{k+1,j}(w).
\end{align*}

Note that, for $j\in\{1,\ldots,k\}$, $\Tsl_{j,k+1}(w) = -\hbar e_{j,k+1}$ and $\Tsl_{k+1,j}(w)
=-\hbar e_{k+1,j}$. Combining these, we arrive at the following expressions:
\begin{align*}
\ev(x_k^+(u)) &= \hbar\Psl_k\lp u+\hhalf\rp^{-1} \sum_{j=1}^k (-1)^{k+j} 
\qmin{1,\ldots,\wh{j},\ldots,k}{1,\ldots,k-1}{u+\hhalf-\hhalf(k-1)} e_{j,k+1}, \\
\ev(x_k^-(u)) &= \hbar\Psl_k\lp u-\hhalf\rp^{-1} \sum_{j=1}^k (-1)^{k+j}
\qmin{1,\ldots,k-1}{1,\ldots,\wh{j},\ldots,k}{u+\hhalf-\hhalf(k-1)} e_{k+1,j}.
\end{align*}

Now $\Psl_k(w)$ commutes with the quantum minors involved in the expressions above.
Moreover, the degree of each quantum minor in the \rhs of the equations is strictly
less than that of $\Psl_k$. Thus, if $\{\roots{k}{1},\ldots,\roots{k}{k}\}$ are the
roots of $\Psl_k(u)$, then we have the following partial fraction decomposition.
\begin{gather*}
\frac{Q_j\lp u+\hhalf-\hhalf(k-1)\rp}{\Psl_k\lp u+\hhalf\rp} = 
\sum_{i=1}^k \frac{1}{u+\hhalf-\roots{k}{i}} \frac{Q_j\lp \roots{k}{i}-\hhalf (k-1)\rp}
{\prod_{c\not= i} (\roots{k}{i}-\roots{k}{c})},\\
\frac{R_j\lp u+\hhalf-\hhalf(k-1)\rp}{\Psl_k\lp u-\hhalf\rp} = 
\sum_{i=1}^k \frac{1}{u-\hhalf-\roots{k}{i}} \frac{R_j\lp \roots{k}{i}+\hbar-\hhalf (k-1)\rp}
{\prod_{c\not= i} (\roots{k}{i}-\roots{k}{c})},
\end{gather*}
where
\begin{gather*}
Q_j(w) = \qmin{1,\ldots,\wh{j},\ldots,k}{1,\ldots,k-1}{w},\\
R_j(w) = \qmin{1,\ldots,k-1}{1,\ldots,\wh{j},\ldots,k}{w}.
\end{gather*}
Note that we have used Proposition \ref{pr:bethe} (3) here, and the following
well--known identity for a rational function vanishing at $\infty$ and whose
denominator has distinct roots ($\deg(p) < r$ in the equation below):
\[
\frac{p(x)}{\prod_{i=1}^r (x-a_i)} = \sum_{i=1}^r \frac{1}{x-a_i} \frac{p(a_i)}{\prod_{j\neq i}
(a_i-a_j)}.
\]
This proves the proposition.
\end{pf}

\subsection{Evaluation homomorphism in $J$--presentation}\label{ssec:compareCP}

It is perhaps worth pointing out that our homomorphism $\ev$ is the evaluation homomorphism
at $0$ from \cite[Prop.~12.1.15]{chari-pressley}, denoted below by $\evcp$. The significant
difference being that $\evcp$ is explicitly given in the $J$--presentation of the
Yangian.\\ 

To see that $\ev = \evcp$ one begins by making the observation that $\Yhsl{n}$
is generated by $\{\xi_{i,0},x_{i,0}^{\pm}\}_{i\in\bfI}$ and $t_{1,1}$ defined as
$t_{1,1} := \xi_{1,1} - \hhalf \xi_{i,0}^2$.
This is because we have the following relations
\[
[t_{1,1},x_{1,r}^{\pm}] = \pm 2 x_{1,r+1}^{\pm} \aand
[t_{1,1},x_{2,r}^{\pm}] = \mp   x_{2,r+1}^{\pm}.
\]

Thus, we can get $\{x_{j,r}^{\pm}\}_{j=1,2}$ from $\{x_{j,0}^{\pm}\}_{j=1,2}$ and $t_{1,1}$.
In turn, using $[x_{2,r}^+,x_{2,s}^-] = \xi_{2,r+s}$ we can obtain $t_{2,1}$. Continuing
in this fashion, we see that every element from $\{x_{i,r}^{\pm}, \xi_{i,r}\}_{i\in\bfI, r\in\N}$
can be written in terms of $\{x_{i,0}^{\pm},\xi_{i,0}\}_{i\in\bfI}$ and $t_{1,1}$.\\

Using the argument given above, and the fact that both $\ev$ and $\evcp$ map
$x_{i,0}^+\mapsto e_{i,i+1}$ and $x_{i,0}^-\mapsto e_{i+1,i}$, we are left with
checking that $\ev(t_{1,1}) = \evcp(t_{1,1})$.\\

\noindent {\em Computation of $\ev(t_{1,1})$.} Recall that we have the following
formula for $\ev(\xi_1(u))$, from \eqref{eq:evxi1} (see also the $n=2$ example from 
Section \ref{ssec:n=2}). Below $C_1 = e_{12}e_{21}+e_{21}e_{12} + \frac{h_1^2}{2}$
is the Casimir of $\sl_2$ corresponding to the node $1$.
\begin{align*}
\ev(\xi_1(u)) &= \frac{\lp u-\hhalf\cowt_2\rp^2 - \frac{\hbar^2}{4}(2C_1+1)}
{\lp u+\hhalf-\hbar\cowt_1\rp \lp u-\hhalf-\hbar\cowt_1\rp}\\
&= \lp 1 - \hbar\cowt_2 u^{-1} - \frac{\hbar^2}{4} (2C_1+1-(\cowt_2)^2 u^{-2}\rp \cdot \\
& \qquad \cdot \lp 1-\hbar\lp \cowt_1-\frac{1}{2}\rp u^{-1}\rp^{-1}
\cdot \lp 1-\hbar\lp \cowt_1-\frac{1}{2}\rp u^{-1}\rp^{-1}.
\end{align*}

Recall that $\xi_{1,1}$ is the coefficient of $\hbar u^{-2}$ in $\xi_1(u)$. A straightforward
computation gives the following answer:

\begin{equation}\label{eq:evt11}
\ev(t_{1,1}) = \hhalf \lp \cowt_2 h_1 - e_{12}e_{21} - e_{21}e_{12}\rp.
\end{equation}

\noindent {\em Computation of $\evcp(t_{1,1})$.}
Combining the expression of $\evcp$ given in \cite[Prop.~12.1.15]{chari-pressley} with
the isomorphism between the $J$--presentation and the loop presentation of $\Yhsl{n}$
from \cite[Thm.~12.1.3]{chari-pressley} (see also \cite{drinfeld-yangian-qaffine}),
we get the following:
\begin{equation}\label{eq:evcpt11}
\evcp(t_{1,1})\! =\! \frac{\hbar}{4} \lp
\sum_{\lambda,\mu} \Tr{}(h_1(I_{\lambda}I_{\mu}\! +\! I_{\mu}I_{\lambda})) I_{\lambda}I_{\mu}
- \sum_{\beta>0} (\beta,\alpha_1) (x^+_{\beta}x^-_{\beta}\! +\! x_{\beta}^-x_{\beta}^+)
\rp
\end{equation}
where
\begin{itemize}
\item $\{I_{\lambda}\}_{\lambda}$ is an orthonormal basis of $\sl_n$ (\wrt the inner
product $(X,Y) = \Tr{}(XY)$ when $X,Y\in\sl_n$ are viewed as $n\times n$ matrices).

\item In the first term, $h_1, I_{\lambda}, I_{\mu}$ are to be multiplied as
$n\times n$ matrices and $\Tr{}$ is the trace of the resulting matrix.

\item $\beta>0$ refers to the set of positive roots of $\sl_n$.
\end{itemize}

We carry out the simplification of the \rhs of \eqref{eq:evcpt11}. Let us write $T_1$
and $T_2$ for the two terms there. Then
\[
T_1 = T_1^0 + \sum_{j>2} \kappa_{1j} - \sum_{j>2} \kappa_{2j} 
\aand
T_2 = 2\kappa_{12} + \sum_{j>2} \kappa_{1j} - \sum_{j>2} \kappa_{2j},
\]
where $\kappa_{ij} = e_{ij}e_{ji} + e_{ij}e_{ji}$ and $T_1^0$ is the Cartan part of
the first term $T_1$, namely when $I_{\lambda},I_{\mu}\in \h$. That is,
\[
T_1^0 =  \sum_{\begin{subarray}{c} \lambda,\mu \\ I_{\lambda},I_{\mu}\in\h \end{subarray}} 
\Tr{}(h_1(I_{\lambda}I_{\mu} + I_{\mu}I_{\lambda})) I_{\lambda}I_{\mu}.
\]
Finally, it is enough to observe that in $U\h$ one has $\displaystyle
\hhalf \cowt_2 h_1 = \frac{\hbar}{4} T_1^0$.


\begin{thebibliography}{10}

\bibitem{alekseev}
A.~Yu. Alekseev, \emph{On {P}oisson actions of compact {L}ie groups on
  symplectic manifolds}, J. Differential Geom. \textbf{45} (1997), no.~2,
  241--256. 

\bibitem{am-gz}
A.~Yu. Alekseev and E.~Meinrenken, \emph{Ginzburg--{W}einstein via
  {G}elfand--{Z}eitlin}, J. Differential Geom. \textbf{76} (2007), no.~1, 1--34.


\bibitem{alekseev-meinrenken}
A.~Yu. Alekseev and E.~Meinrenken, \emph{Linearization of {P}oisson--{L}ie group
  structures}, J. Symplectic Geom. \textbf{14} (2016), no.~1, 227--267.

\bibitem{andrea-sachin-2}
A.~Appel and S.~Gautam, \emph{{Quantization of Alekseev--Meinrenken diffeomorphism}},
in preparation.

\bibitem{atl0}
A.~Appel and V.~Toledano~Laredo, \emph{{Monodromy of the Casimir connection of
  a symmetrisable Kac--Moody algebra}},  (2015), \href{https://arxiv.org/abs/1512.03041}{\sf arXiv:1512.03041}.

\bibitem{atl1}
A.~Appel and V.~Toledano~Laredo, \emph{Uniqueness of Coxeter structures on Kac--Moody algebras}, 
Adv. Math. {\bf 347} (2019), 1--104.

\bibitem{boalch-stokes}
P.~P. Boalch, \emph{Stokes matrices, {P}oisson {L}ie groups and {F}robenius
  manifolds}, Invent. Math. \textbf{146} (2001), no.~3, 479--506. 

\bibitem{brundan-kleshev}
J.~Brundan and A.~Kleshchev, \emph{Parabolic presentations of the {Y}angian
  ${Y}(\gl_n)$}, Comm. Math. Phys. \textbf{254} (2005), 191--220.

\bibitem{chari-pressley}
V.~Chari and A.~Pressley, \emph{A guide to quantum groups}, Cambridge
  University Press, 1994.

\bibitem{dixmier}
J.~Dixmier, \emph{Enveloping algebras}, Graduate Studies in Mathematics,
  vol.~11, American Mathematical Society, Providence, RI, 1996. Revised reprint
  of the 1977 translation. 
  
\bibitem{dskv17}
A. De Sole, V. G. Kac\ and\ D. Valeri, \emph{A Lax type operator for quantum finite $W$-algebras}, 
Selecta Math. (N.S.) {\bf 24} (2018), no.~5, 4617--4657. 

\bibitem{drinfeld-qybe}
V.~G. Drinfeld, \emph{{Hopf algebras and the quantum Yang-Baxter equation}},
  Soviet Math. Dokl. \textbf{32} (1985), no.~1, 254--258.

\bibitem{drinfeld-quantum-groups}
V.~G. Drinfeld, \emph{Quantum groups}, Proceedings of the I.C.M., Berkeley (1986),
  798--820.

\bibitem{drinfeld-yangian-qaffine}
V.~G. Drinfeld, \emph{{A new realization of Yangians and quantum affine algebras}},
  Soviet Math. Dokl. \textbf{36} (1988), no.~2, 212--216.

\bibitem{drinfeld-almost}
V.~G. Drinfeld, \emph{On almost cocommutative {H}opf algebras}, Leningrad Math. J.
  \textbf{1} (1990), no.~2, 321--342.

\bibitem{enriquez-etingof-marshall}
B.~Enriquez, P.~Etingof, and I.~Marshall, \emph{Comparison of {P}oisson
  structures and {P}oisson--{L}ie dynamical {$r$}-matrices}, Int. Math. Res.
  Not. (2005), no.~36, 2183--2198.

\bibitem{sachin-valerio-1}
S.~Gautam and V.~Toledano~Laredo, \emph{Yangians and quantum loop algebras},
  Selecta Math. (N.S.) \textbf{19} (2013), no.~2, 271--336.

\bibitem{sachin-valerio-2}
S.~Gautam and V.~Toledano~Laredo, \emph{Yangians, quantum loop algebras, and abelian difference
  equations}, J. Amer. Math. Soc. \textbf{29} (2016), no.~3, 775--824.

\bibitem{gavarini}
F.~Gavarini, \emph{The quantum duality principle}, Ann. Inst. Fourier
  (Grenoble) \textbf{52} (2002), no.~3, 809--834. 


\bibitem{ginzburg-weinstein}
V.~L. Ginzburg and A.~Weinstein, \emph{Lie--{P}oisson structure on some
  {P}oisson {L}ie groups}, J. Amer. Math. Soc. \textbf{5} (1992), no.~2,
  445--453. 

\bibitem{jimbo-qg}
M.~Jimbo, \emph{A {$q$}-difference analogue of {$U(\mathfrak{g})$} and the
  {Y}ang-{B}axter equation}, Lett. Math. Phys. \textbf{10} (1985), no.~1,
  63--69. 

\bibitem{jimbo-R-matrix}
M.~Jimbo, \emph{{Quantum $R$--matrix for the generalized Toda system}}, Comm. Math.
  Phys. \textbf{102} (1986), 537--547.

\bibitem{levendorskii}
S.~Z. Levendorskii, \emph{{On generators and defining relations of Yangians}},
  Journal of Geometry and Physics \textbf{12} (1992), 1--11.

\bibitem{molev-yangian}
A.~Molev, \emph{Yangians and classical {L}ie algebras}, Mathematical Surveys
  and Monographs, vol. 143, A.M.S., 2007.

\bibitem{nazarov-tarasov-rep-gt}
M.~Nazarov and V.~Tarasov, \emph{Representations of {Y}angians with
  {G}elfand-{Z}etlin bases}, J. Reine Angew. Math. \textbf{496} (1998),
  181--212.

\bibitem{valerio0}
V.~Toledano Laredo, \emph{A Kohno-Drinfeld theorem for quantum Weyl groups}, Duke Math. J. {\bf 112} (2002), no.~3, 421--451.

\bibitem{valerio1}
V.~Toledano Laredo, \emph{Quasi--{C}oxeter algebras, {D}ynkin diagram cohomology and
  quantum {W}eyl groups}, Int. Math. Res. Pap. \textbf{2008} (2008), 167 pp.

\end{thebibliography}
\end{document}